\documentclass[12pt]{amsart}

\usepackage{amsmath}
\usepackage{amsfonts}
\usepackage{amssymb}
\usepackage{latexsym}

\usepackage[all]{xy}

\usepackage{amsmath}
\usepackage{amsfonts}
\usepackage{amssymb}
\usepackage{latexsym}
\usepackage[all]{xy}
\newtheorem{theorem}{Theorem}[section]
\newtheorem{proposition}[theorem]{Proposition}
\newtheorem{corollary}[theorem]{Corollary}

\theoremstyle{remark}
\newtheorem{remark}[theorem]{Remark}
\theoremstyle{definition}

\numberwithin{equation}{section}

\def\sD{{\mathfrak D}}      
   \def\sH{{\mathfrak H}}   
   \def\sK{{\mathfrak K}}   
\def\sM{{\mathfrak M}}   \def\sN{{\mathfrak N}}

      \def\dC{{\mathbb C}}
\def\dD{{\mathbb D}}

   \def\dN{{\mathbb N}}   
      
   \def\dT{{\mathbb T}}

   \def\cH{{\mathcal H}}   
   \def\cK{{\mathcal K}}   \def\cL{{\mathcal L}}
\def\cM{{\mathcal M}}      
\def\cP{{\mathcal P}}

\def\bL{{\mathbf L}}
\def\bS{{\mathbf S}}

\def\clos{{\rm clos\,}}
\def\RE{{\rm Re\,}}

\def\wt{\widetilde}
\def\wh{\widehat}

\def\f{\varphi}
\def\uphar{{\upharpoonright\,}}

\def\ran{{\rm ran\,}}
\def\dom{{\rm dom\,}}
\def\cran{{\rm \overline{ran}\,}}
\def\cspan{{\rm \overline{span}\, }}
\begin{document}
\title
 {Schur parameters, Toeplitz matrices, and Kre\u{\i}n shorted operators}

\author{
Yury~Arlinski\u{i}}
\address{Department of Mathematical Analysis \\
East Ukrainian National University \\
Kvartal Molodyozhny 20-A \\
Lugansk 91034 \\
Ukraine} \email{yury$\_$arlinskii@yahoo.com}
 \subjclass%[2010]
 {47A48, 47A56, 47A57, 47A64, 47B35, 47N70, 93B05, 93B07}


\thispagestyle{empty}

\begin{abstract}
We establish connections between Schur parameters of the Schur class
operator-valued functions, the corresponding simple conservative
realizations, lower triangular Toeplitz matrices, and Kre\u\i n
shorted operators. By means of Schur parameters or shorted operators
for defect operators of Toeplitz matrices necessary and sufficient
conditions for a simple conservative discrete-time system to be
controllable/observable and for a completely non-unitary contraction
to be completely non-isometric/completely non-co-isometric are
obtained. For the Schur problem a characterization of central
solution and uniqueness criteria to the solution are given in terms
of shorted operators for defect operators of contractive Toeplitz
matrices, corresponding to data.
\end{abstract}
\maketitle

\section{Introduction}
In this Section we briefly describe notations, the basic objects,
and the main goal of this paper.
\subsection{Notations}
In what follows the class of all continuous linear operators defined
on a complex Hilbert space $\sH_1$ and taking values in a complex
Hilbert space $\sH_2$ is denoted by $\bL(\sH_1,\sH_2)$ and
${\bL}(\sH):= {\bL}(\sH,\sH)$. All infinite dimensional Hilbert
spaces are supposed to be separable. We denote by $I$ the identity
operator in a Hilbert space and by $P_\cL$ the orthogonal projection
onto the subspace (the closed linear manifold) $\cL$. The notation
$T\uphar \cL$ means the restriction of a linear operator $T$ on the
set $\cL$. The range and the null-space of a linear operator $T$ are
denoted by $\ran T$ and $\ker T$, respectively. We use the usual
symbols $\dC$, $\dN$, and $\dN_0$ for the sets of complex numbers,
positive integers, and nonnegative integers, respectively. The
\textit{Schur class} ${\bf S}(\sH_1,\sH_2)$ is the set of all
function $\Theta(\lambda)$ analytic on the unit disk
$\dD=\{\lambda\in\dC:|\lambda|<1\}$ with values in
$\bL(\sH_1,\sH_2)$ and such that $\|\Theta(\lambda)\|\le 1$ for all
$\lambda\in\dD$. An operator $T\in\bL(\sH_1,\sH_2)$ is said to be
\begin{itemize}
\item \textit{contractive} if $\|T\|\le 1$;

\item \textit{isometric} if $\|Tf\|=\|f\|$ for all $f\in \sH_1$
$\iff T^*T=I$;

\item \textit{co-isometric} if $T^*$ is isometric $\iff
TT^*=I$;
\item \textit{unitary} if it is both isometric and co-isometric.
\end{itemize}
Given a contraction $T\in \bL(\sH_1,\sH_2)$, the operators
$
D_T:=(I-T^*T)^{1/2}$ and $D_{T^*}:=(I-TT^*)^{1/2}
$
are called the \textit{defect operators} of $T$, and the subspaces
$\sD_T=\cran D_T,$ $\sD_{T^*}=\cran D_{T^*}$ the \textit{defect
subspaces} of $T$. %The dimensions $\dim\sD_T,$ $\dim\sD_{T^*}$ are known as the \textit{defect numbers} of $T$.
The defect operators
satisfy the following relations
$ TD_{T}=D_{T^*}T,$ $T^*D_{T^*}=D_{T}T^*.$

\subsection{The Schur algorithm}
Given a scalar Schur class function $f(\lambda)$, which is not a finite Blaschke
product, define inductively
\[
f_0(\lambda)=f(\lambda),\;
f_{n+1}(\lambda)=\frac{f_n(\lambda)-f_n(0)}{\lambda(1-\overline{f_n(0)}f_n(\lambda))},\;
n\in\dN_0.
\]
It is clear that  $\{f_n\}$ is an {\it infinite} sequence of Schur
functions called the \textit{associated functions} and neither of
its terms is a finite Blaschke product. The numbers
$\gamma_n:=f_n(0)$ are called the {\it Schur parameters}.
Note that
\[
f_n(\lambda)=\frac{\gamma_n+\lambda
f_{n+1}(\lambda)}{1+\bar\gamma_n\lambda
f_{n+1}}=\gamma_n+(1-|\gamma_n|^2)\frac{\lambda
f_{n+1}(\lambda)}{1+\bar\gamma_n\lambda f_{n+1}(\lambda)},\;
n\in\dN_0.
\]
The method of labeling $f\in{\bf S}$ by its Schur parameters is
known as the \textit{Schur algorithm} and is due to I.~Schur
\cite{Schur}. In the case when $f$
%\[f(\lambda)=e^{i\f}\prod_{k=1}^N \frac{\lambda-\lambda_k}{1-\bar\lambda_k \lambda}\]
is a finite Blaschke product of order $N$, the Schur algorithm
terminates at the $N$-th step, i.e., the sequence of Schur parameters
$\{\gamma_n\}_{n=0}^N$ is finite, $|\gamma_n|<1$ for
$n=0,1,\ldots,N-1$, and $|\gamma_N|=1$.

The next theorem goes back to Shmul'yan \cite{Shmul1, Shmul2} and
T.~Constantinescu \cite{Const} (see also \cite{ARL1, BC, Const2,
DGK, DFK}) and plays a key role in the Schur algorithm for
operator-valued functions.
\begin{theorem}
%\label{MO}
Let $\sM$ and $\sN$ be Hilbert spaces and let the function
$\Theta(\lambda)$ be from the Schur class ${\bf S}(\sM,\sN).$ Then
there exists a function $Z(\lambda)$ from the Schur class ${\bf
S}(\sD_{\Theta(0)},\sD_{\Theta^*(0)})$ such that
\begin{equation}
\label{MREP}
\Theta(\lambda)=\Theta(0)+D_{\Theta^*(0)}Z(\lambda)(I+\Theta^*(0)Z(\lambda))^{-1}D_{\Theta(0)},\;\lambda\in\dD.
\end{equation}
\end{theorem}
The representation \eqref{MREP} of a function $\Theta(\lambda)$ from
the Schur class is called \textit{the M\"obius representation} of
$\Theta(\lambda)$ and the function $Z(\lambda)$ is called
\textit{the M\"obius parameter} of $\Theta(\lambda)$. Clearly,
$Z(0)=0$ and from Schwartz's lemma one obtains that
\[
\lambda^{-1}Z(\lambda)\in\bS(\sD_{\Theta(0)},\sD_{\Theta^*(0)}).
\]
\textit{The operator Schur's algorithm} \cite{BC}. For
$\Theta\in{\bf S}(\sM,\sN)$ put
$\Theta_0(\lambda)=\Theta(\lambda)$ and let $Z_0(\lambda)$ be the
M\"obius parameter of $\Theta$. Define
\[
\Gamma_0=\Theta(0),\;
\Theta_1(\lambda)=\lambda^{-1}Z_0(\lambda)\in {\bf
S}(\sD_{\Gamma_0},\sD_{\Gamma^*_0}),\;\Gamma_1= \Theta_1(0)=Z'_0(0).
\]
If $\Theta_0(\lambda),\ldots,\Theta_n(\lambda)$ and
$\Gamma_0,\ldots, \Gamma_n$ have been chosen, then let $Z_{n+1}\in
{\bf S}(\sD_{\Gamma_n},\sD_{\Gamma^*_n})$ be the M\"obius parameter
of $\Theta_n$. Put
\[
\Theta_{n+1}(\lambda)=\lambda^{-1}Z_{n+1}(\lambda),\;
\Gamma_{n+1}=\Theta_{n+1}(0).
\]
 The contractions $\Gamma_0\in\bL(\sM,\sN),$
$\Gamma_n\in\bL(\sD_{\Gamma_{n-1}},\sD_{\Gamma^*_{n-1}})$,
$n=1,2,\ldots$ are called the \textit{Schur parameters} of
$\Theta(\lambda)$ and the function $\Theta_n \in {\bf
S}(\sD_{\Gamma_{n-1}},\sD_{\Gamma^*_{n-1}})$ is called the $n-th$
\textit{associated function}. Thus,
\[
\Theta_n(\lambda)=\Gamma_n+\lambda D_{\Gamma^*_n}\Theta_{n+1}(\lambda)(I+\lambda\Gamma^*\Theta_{n+1}(\lambda))^{-1}D_{\Gamma_n},\;\lambda\in\dD,
\]
and
\[
\Theta_{n+1}(\lambda)\uphar\ran
D_{\Gamma_n}=\lambda^{-1}D_{\Gamma^*_n}(I-\Theta_n(\lambda)\Gamma^*_n)^{-1}
(\Theta_n(\lambda)-\Gamma_n)D^{-1}_{\Gamma_n}\uphar\ran
D_{\Gamma_n}.
\]
Clearly, the sequence of Schur parameters $\{\Gamma_n\}$ is infinite
if and only if the operators $\Gamma_n$ are non-unitary. The
sequence of Schur parameters consists of finite number of operators
$\Gamma_0,$ $\Gamma_1,\ldots, \Gamma_N$ if and only if
$\Gamma_N\in\bL(\sD_{\Gamma_{N-1}},\sD_{\Gamma^*_{N-1}})$ is
unitary. If $\Gamma_N$ is non-unitary but isometric (respect.,
co-isometric), then $\Gamma_n=0\in\bL(0,\sD_{\Gamma^*_N})$
(respect., $\Gamma_n=0\in\bL(\sD_{\Gamma_N},0)$) for all $n>N$. The
following theorem is the operator generalization of Schur's result.
\begin{theorem} \label{SchurAlg}\cite{BC, Const}. There is a one-to-one
correspondence between the Schur class ${\bf S}(\sM,\sN)$ and the
set of all sequences of contractions $\{\Gamma_n\}_{n\ge 0}$ such
that
\begin{equation}
\label{CHSEQ} \Gamma_0\in\bL(\sM,\sN),\;\Gamma_n\in
\bL(\sD_{\Gamma_{n-1}},\sD_{\Gamma^*_{n-1}}),\; n\ge 1.
\end{equation}
\end{theorem}
Notice that a sequence of contractions of the form \eqref{CHSEQ} is
called the \textit{choice sequence} \cite{CF}.

\subsection{The lower triangular Toeplitz matrices}

Let $\Theta$ be holomorphic in $\dD$ operator valued function acting between Hilbert spaces $\sM$ and $\sN$ and let
\[
\Theta(\lambda)=\sum\limits_{n=0}^{\infty}\lambda^n C_n,\; \lambda\in\dD,\; C_n\in\bL(\sM,\sN), n\ge 0
\]
be the Taylor expansion of $\Theta$. Consider the lower triangular
(analytic) Toeplitz matrix
\begin{equation}
\label{toep} T_\Theta:=\begin{bmatrix}C_0&0&0&0&\ldots&\ldots\cr
C_1&C_0&0&0&\ldots&\ldots\cr
C_2&C_1&C_0&0&0&\ldots\cr
C_3&C_2&C_1&C_0&0&\ldots\cr
\vdots&\vdots&\vdots&\vdots&\vdots&\vdots%\cr
%C_n&C_{n-1}&C_{n-2}&C_{n-3}&\ldots\cr
%\vdots&\vdots&\vdots&\vdots&\vdots
\end{bmatrix}.
\end{equation}
%We shall not distinguish between matrices and corresponding linear operators.
As is well known \cite{BC, FoFr}
\[
\Theta\in\bS(\sM,\sN)\iff
T_\Theta\in\bL\left(l_2(\sM),l_2(\sN)\right)\quad\mbox{is a
contraction}.
\]
Set for $n=0,1,\ldots$
\[\begin{array}{l}
\sM^{n+1}=\underbrace{\sM\oplus\sM\oplus\cdots\oplus\sM}_{n+1},\; %\\
\sN^{n+1}=\underbrace{\sN\oplus\sN\oplus\cdots\oplus\sN}_{n+1}.\\
\end{array}
\]
 Clearly, if $T_\Theta$ is a contraction, then the
operator $T_{\Theta,n}\in \bL\left(\sM^{n+1},\sN^{n+1}\right)$ given
by the block operator matrix
\begin{equation}
\label{toepn} T_{\Theta,n}:=\begin{bmatrix}C_0&0&0&\ldots&0\cr
C_1&C_0&0&\ldots&0\cr \vdots&\vdots&\vdots&\vdots&\vdots\cr
C_n&C_{n-1}&C_{n-2}&\ldots&C_0\end{bmatrix}
\end{equation}
is a contraction  for each $n$. There are connections, established
by T.~Constantinescu \cite{Const}, between the Taylor coefficients
$\{C_n\}_{n\ge 0}$ and Schur parameters of $\Theta\in\bS(\sM,\sN)$.
These connections are given by the relations
\begin{equation}
\label{ConstForm}
\begin{array}{l}
C_0=\Gamma_0,\\
 C_n={formula}_n(\Gamma_0,\Gamma_1,\cdots,
\Gamma_{n-1})+\\
\qquad \qquad D_{\Gamma^*_0}D_{\Gamma^*_1}\cdots D_{\Gamma^*_{n-1}}
\Gamma_n D_{\Gamma_{n-1}}\cdots D_{\Gamma_1}D_{\Gamma_0},\; n\ge 1.
\end{array}
\end{equation}
 Here ${formula}_n(\Gamma_0,\Gamma_1,\cdots, \Gamma_{n-1})$
is a some expression, depending on $\Gamma_0,\Gamma_1,\cdots,
\Gamma_{n-1}$.

Let now $\{C_k\}_{k=0}^\infty$ be a sequence of operators from $\bL(\sM,\sN)$.
Then (\cite[Theorem 2.1]{BC}) there is a one-to-one correspondence between the set of contractions
\[%begin{equation}
%\label{toep1}
 T_\infty:=\begin{bmatrix}C_0&0&0&0&0&\ldots\cr
C_1&C_0&0&0&0&\ldots\cr
C_2&C_1&C_0&0&0&\ldots\cr
C_3&C_2&C_1&C_0&0&\ldots\cr
\vdots&\vdots&\vdots&\vdots&\vdots&\vdots%\cr
%C_n&C_{n-1}&C_{n-2}&C_{n-3}&\ldots\cr
%\vdots&\vdots&\vdots&\vdots&\vdots
\end{bmatrix}: l_2(\sM)\to l_2(\sN)
\]
and the set of all choice sequences $\Gamma_0\in\bL(\sM,\sN),$ $\Gamma_k\in\bL(\sD_{\Gamma_{k-1}},\sD_{\Gamma^*_{k-1}})$, $k=1,\ldots$.
The connections between $\{C_k\}$ and $\{\Gamma_k\}$ is also given by \eqref{ConstForm}.
The operators $\{\Gamma_k\}$ can be by successively defined \cite[proof of Theorem 2.1]{BC},
 using parametrization of contractive
block-operator matrices (see Section  \ref{parcontmatr}), from the matrices
\[
T_0=C_0=\Gamma_0, \;T_1=\begin{bmatrix}C_0&0\cr C_1&C_0\end{bmatrix}, \;T_2=\begin{bmatrix}C_0&0&0\cr C_1&C_0&0\cr
C_2&C_1&C_0\end{bmatrix},\ldots.
\]
 Moreover, $T_\infty=T_\Theta$,  $\Theta(\lambda)=\sum\limits_{n=0}^{\infty}\lambda^n C_n$, $\lambda\in\dD$, and $\{\Gamma_k\}_{k\ge 0}$
are the Schur parameters of $\Theta$ \cite[Proposition 2.2]{BC}. Put
\[
\wt\Theta(\lambda):=\Theta^*(\bar \lambda),\;|\lambda|<1.
\]
Then $ \wt\Theta(\lambda)=\sum\limits_{n=0}^{\infty}\lambda^n
C^*_n.$ Clearly, if $\{\Gamma_0,\Gamma_1,\ldots\}$ are the Schur
parameters of $\Theta$, then $\{\Gamma^*_0,\Gamma^*_1,\ldots\}$ are
the Schur parameters of $\wt \Theta$.

\subsection{The Schur problem}\label{SchPr}
The following problem is called the \textit{Schur problem}:

\textit{Let $\sM$ and $\sN$ be Hilbert spaces. Given the operators
$C_k\in\bL(\sM,\sN)$, $k=0,1,\ldots, N$, it is required to
({\bf{a}}) find conditions for the existence of
$\Theta\in\bS(\sM,\sN)$ such that $C_0,C_1,\ldots, C_N$ are the
first $N+1$ Taylor coefficients of $\Theta$,({\bf{b}}) give an
explicit description of all solutions $\Theta$ (if there any) to
problem ({\bf{a}}).}

The Schur problem is often called the Carath\'eodory or the
Carath\'eodory-Fej\'{e}r problem. This problem was studied in many
papers, see monographs \cite{BC, DFK, FoFr} and references therein.
It is well known that the Schur problem has a solution if and only
if the Toeplitz operator from $\bL(\sM^{N+1},\sN^{N+1})$
\begin{equation}
\label{trToe} T_N=T_N(C_0,C_1,\ldots,
C_N):=\begin{bmatrix}C_0&0&0&\ldots&0\cr C_1&C_0&0&\ldots&0\cr
\vdots&\vdots&\vdots&\vdots&\vdots\cr
C_N&C_{N-1}&C_{N-2}&\ldots&C_0\end{bmatrix}
\end{equation}
is a contraction. By means of relations \eqref{ConstForm}
contractions $T_0, T_1,\ldots, T_N$ determine choice parameters
\[
\Gamma_0:=C_0,\; \Gamma_1\in\bL(\sD_{\Gamma_0},\sD_{\Gamma_0^*}),\ldots,\Gamma_N\in\bL(\sD_{\Gamma_{N-1}},\sD_{\Gamma_{N-1}^*})
\]
If $T_N$ is a contraction, then operators $\{C_k\}_{k=0}^N$ are said
to be \textit{the Schur sequence} \cite{DFK}. Let us formulate known
conditions for a uniqueness solution to the Schur problem.
\begin{theorem}
\label{SchPr1} \cite[Proposition 2.3]{BC}. Let the complex numbers
$\{C_k\}_{k=0}^N$ be the Schur sequence. Then the following
assertions are equivalent:
\begin{enumerate}
\def\labelenumi{\rm (\roman{enumi})}
\item the Schur problem with data $\{C_k\}_{k=0}^N$ has a unique solution;
\item there exists a number $r$, $0\le r\le N$ such that $|\Gamma_r|=1$;
%\item $\rm{rank} D^2_{T_N}=r;$
\item $\det D^2_{T_r}=0$ for some $0\le r\le N$, but $\det D^2_{T_p}\ne 0$ for $0\le p<r;$
\item $\det D^2_{T_N}=0$.
\end{enumerate}
\end{theorem}
\begin{theorem}
\label{SchPr2} \cite[Theorem 2.6]{BC}. Consider a solvable Schur problem with the data
$$C_0,\ldots, C_N \in\bL(\sM,\sN).$$
Then the solution is unique if and only if the corresponding choice parameters $\{\Gamma_n\}_{n=0}^N$, determined by the operator $T_N$, satisfy
the condition: one of $\Gamma_n$, $0\le n\le N$ is an isometry or a co-isometry.
\end{theorem}
\subsection{Simple conservative discrete time-invariant systems and their transfer functions}
Here we recall some results from the theory of conservative discrete
time-invariant systems cf. \cite{Alpay, ADRS, A, Arov, Br1, Helton2,
Ball-Coehn, Staf1}.

A collection
\[
\tau=\left\{\begin{bmatrix}D&C\cr B& A\end{bmatrix}; \mathfrak M,\mathfrak N,\mathfrak H\right\}
\]
is called the linear discrete time-invariant systems with the state space $\mathfrak H$ and the input and output
spaces $\mathfrak M$ and $\mathfrak N $, respectively. A system $\tau$ is called conservative if the
linear operator
\[%begin{equation}
%\label{sys}
T_\tau=\begin{bmatrix} D&C \cr B&A\end{bmatrix} :
\begin{array}{l} \sM \\\oplus\\ \sH \end{array} \to
\begin{array}{l} \sN \\\oplus\\ \sH \end{array}
\]%end{equation}
is unitary.
The \textit{transfer function}
\[%begin{equation}
%\label{TrFu}
\Theta_\tau(\lambda):=D+\lambda C(I-\lambda A)^{-1}B, \quad \lambda
\in \dD,
\]%end{equation}
of a conservative system $\tau$ belongs to the Schur class ${\bf
S}(\sM,\sN)$.
Conservative systems are also called unitary colligations  and their transfer
functions are called the characteristic functions \cite{Br1}.
The subspaces
\[%begin{equation}
%\label{CO}
\sH^c_\tau:=\cspan\{A^{n}B\sM:\,n=0,1,\ldots\} \mbox{ and
} \sH^o_\tau=\cspan\{A^{*n}C^*\sN:\,n=0,1,\ldots\}
\]%end{equation}
are said to be the \textit{controllable} and \textit{observable}
subspaces of the system $\tau$, respectively. The system $\tau$ is
said to be \textit{controllable} (respect., \textit{observable}) if
$\sH^c_\tau=\sH$ (respect., $\sH^o_\tau=\sH$), and it is called
\textit{minimal} if $\tau$ is both controllable and observable. The
system $\tau$ is said to be \textit{simple} if $\sH=\clos
\{\sH^c_\tau+\sH^o_\tau\}$ (the closure of the span).
Two discrete time-invariant systems
\[
\tau_1=\left\{\begin{bmatrix} D&C_1 \cr
B_1&A_1\end{bmatrix};\sM,\sN,\sH_{1}\right\} \quad \mbox{and} \quad
\tau_2=\left\{\begin{bmatrix} D&C_2 \cr
B_2&A_2\end{bmatrix};\sM,\sN,\sH_{2}\right\}
\]
are said to be \textit{unitarily similar} if there exists a unitary
operator $U$ from $\sH_{1}$ onto $\sH_{2}$ such that
\[
A_1 =U^{-1}A_2U,\quad B_1=U^{-1}B_2,\quad C_1=C_2 U.
\]
As is well known, two simple conservative systems with the same
transfer function are unitarily similar. It is important that any
function $\Theta\in {\bf S}(\sM,\sN)$ can be realized as the
transfer function of a linear conservative and simple discrete-time
system.

\subsection{M.~Kre\u{\i}n's shorted operators} For every nonnegative bounded operator
$S$ in the Hilbert space $\cH$ and every subspace $\cK\subset \cH$
M.G.~Kre\u{\i}n \cite{Kr} defined the operator $S_{\cK}$ by the
relation
\[
 S_{\cK}=\max\left\{\,Z\in \bL(\cH):\,
    0\le Z\le S, \, {\ran}Z\subseteq{\cK}\,\right\}.
\]
The equivalent definition
\begin{equation}
\label{Sh1}
 \left(S_{\cK}f, f\right)=\inf\limits_{\f\in \cK^\perp}\left\{\left(S(f + \varphi),f +
 \varphi\right)\right\},
\quad  f\in\cH.
\end{equation}
Here $\cK^\perp:=\cH\ominus{\cK}$. The properties of $S_{\cK}$, were
studied by M.~Kre\u{\i}n and by other authors (see \cite{ARL1} and
references therein).
 $S_{\cK}$ is called the \textit{shorted
operator} (see \cite{And, AT}). Let the subspace $\Omega$ be defined
as follows
\[
 \Omega=\{\,f\in \cran S:\,S^{1/2}f\in {\cK}\,\}=\cran S\ominus
S^{1/2}\cK^\perp.
\]
It is proved in \cite{Kr} that $S_{\cK}$ takes the form
\[
 S_{\cK}=S^{1/2}P_{\Omega}S^{1/2}.
\]
 Hence, $\ker S_\cK\supseteq\cK^\perp.$
Moreover \cite{Kr},
\begin{equation}
\label{rangeSh}
 {\ran}S_{\cK}^{1/2}={\ran}S^{1/2}\cap{\cK}.
\end{equation}
It follows that
\begin{equation}
\label{nol}
  S_{\cK}=0 \iff  \ran S^{1/2}\cap \cK=\{0\}.
\end{equation}
\subsection{The goal of this paper}
In this paper we establish connections between the Schur parameters
of $\Theta\in\bS(\sM,\sN)$, a simple conservative realization of
$\Theta,$ the operators $T_\Theta$ and $T_{\Theta,n}$, and the
Kre\u{\i}n shorted operators. These connections allows to
\begin{enumerate}
\item give criterions of controllability and observability for the corresponding to $\Theta$ simple conservative system in terms of Schur parameters/
Kre\u\i n shorted operators $\left(D^2_{T_\Theta}\right)_\sM$ and
$\left(D^2_{T_{\wt\Theta}}\right)_\sN$,
\item to obtain necessary and sufficient conditions for a completely non-unitary contraction $A$ to be completely non-isometric
 or completely non-co-isometric  \cite{Ball-Coehn} in
terms of Schur parameters / Kre\u\i n shorted operators
$\left(D^2_{T_\Psi}\right)_{\sD_A}$,
$\left(D^2_{T_{\wt\Psi}}\right)_{\sD_{A^*}}$ of Sz-Nagy--Foias
characteristic function $\Psi$ of $A$ \cite{SF},
\item give a characterization of the central (maximal entropy)
solution to the Schur problem,
\item give a uniqueness criterion to the solution of the operator Schur problem in terms of the Kre\u{\i}n shorted operators for the defect operators
of the Toeplitz matrices, constructed from problem's data.
\end{enumerate}
The paper is organized as follows. Sections \ref{parcontmatr},
\ref{secS}, \ref{CRSCAL} deal with additional background material
concerning parametrization of $2\times 2$ contractive and unitary
block operator matrices, the theory of completely non-unitary
contractions, defect functions of
 the Schur class functions, and conservative realization of the Schur algorithm.
New results about the Kre\u\i n shorted operators are given in
Section \ref{KRSHOP}. Main results of the paper are presented in
Section \ref{MARE}. Relying on the results of Section \ref{KRSHOP},
we prove that the Kre\u\i n shorted operators
$\left\{\left(D^2_{T_k}\right)_\sM\uphar\sM\right\}$ forms a
non-increasing sequence, where $T_k=T_k(C_0, C_1,\ldots C_k)$ are
the Toeplitz operators constructed from the Schur sequence. We study
in more detail the central solution to the Schur problem and obtain
a uniqueness solution criteria. The latter is closed to results of
V.M.~Adamyan, D.Z.~Arov, and M.G.~Kre\u\i n obtained in \cite{AAK1}
and \cite{AAK2} concerning to scalar and operator Nehari problem
\cite{Nehari}. These authors did not use the Kre\u\i n shorted
operators in explicit form, their approach is essentially rely on
the extension theory of isometric operators. Different approaches to
the descriptions of all  solutions to the Schur problem can be found
in \cite{DFK} for finite dimensional $\sM$ and $\sN$, in \cite{BC,
FoFr} for general case. The Schur problem can be reduced to the
above mentioned Nehari problem \cite{ArovKrein}. All solutions to
this problem are obtained in \cite{AAK1, AAK2, ArovKrein, Kh} (see
also \cite{Peller}).

\section{Parametrization of contractive block-operator matrices}\label{parcontmatr}
Let $\sH,\,\sK,$ $\sM$ and $\sN$ be Hilbert spaces. The following
theorem goes back to \cite{AG, DaKaWe, ShYa}; other proofs of the
theorem can be found in \cite{ARL, AHS1, KolMal, Mal2, Peller}.
\begin{theorem} \label{ParContr1}
Let $A \in \bL(\cH,\cK)$, $B \in \bL(\sM,\cK)$, $C \in
\bL(\cH,\sN)$, and $D \in \bL(\sM,\sN)$.
The following conditions are equivalent:
\begin{enumerate}
\def\labelenumi{\rm (\roman{enumi})}
\item the operator
$
T= \begin{bmatrix} D&C\cr B&A \end{bmatrix}:
\begin{array}{l}\sM\\\oplus\\\cH\end{array}\to
\begin{array}{l}\sN\\\oplus\\\cK\end{array}
$
is a contraction;
\item the operator $A \in \bL(\cH,\cK)$ is a contraction and
\begin{equation}\label{twee}
%\begin{bmatrix}
B=D_{A^*}M,\; C=KD_{A},\;D=-KA^*M+D_{K^*}XD_{M},
\end{equation}
where $M\in\bL(\sM,\sD_{A^*})$,
$K\in\bL(\sD_{A},\sN)$, and $X\in\bL(\sD_{M},\sD_{K^*})$
are contractions;
\item  the operator $D\in\bL(\sM,\sN)$ is a contraction and
\begin{equation}
\label{BLOCKS}
B=FD_D,\; C=D_{D^*}G,\; A=-FD^*G+D_{F^*}LD_{G}.
\end{equation}
where the operators $F\in\bL(\sD_D,\cK)$, $G\in\bL(\cH,\sD_{D^*})$
and $L\in\bL(\sD_{G},\sD_{F^*})$ are contractions.
\end{enumerate}
Moreover, if $T$ is a contraction, then the operators $K,$ $M$, and $X$ in \eqref{twee} and
operators $F,\,G,$ and $L$ in \eqref{BLOCKS} are uniquely determined.
\end{theorem}
\begin{corollary}
\label{iscois1} \cite{ARL1}, \cite{OAM2009}. Let
\[
T=
\begin{bmatrix}-KA^*M+D_{K^*}XD_{M} &KD_{A} \cr
D_{A^*}M&A\end{bmatrix}=\begin{bmatrix} D&D_{D^*}G\cr FD_D&-FD^*G+D_{F^*}LD_G
\end{bmatrix}:
\begin{array}{l}\sM\\\oplus\\\cH\end{array}\to
\begin{array}{l}\sN\\\oplus\\\cK\end{array}
\]
be a contraction. Then
\begin{enumerate}
\item
$(D^2_T)_\sM=D_MD^2_XD_MP_\sM,\; (D^2_{T^*})_\sN=D_{K^*}D^2_{X^*}D_{K^*}P_\sN,$
\item $(D^2_T)_\cH=D_GD^2_LD_GP_\cH,\;  (D^2_{T^*})_\cK=D_{F^*}D^2_{L^*}D_{F^*}P_\cK,$
\item $T$ is isometric if
and only if
\[
D_KD_A=0,\;D_XD_M=0,\;D_FD_D=0,\; D_LD_G=0,
\]
\item
$T$ is co-isometric if and only if
\[
D_{M^*}D_{A^*}=0,\; D_{X^*}D_{K^*}=0,\;D_{G^*}D_{D^*}=0,\; D_{L^*}D_{F^*}=0.
\]
\end{enumerate}
If $T$ is unitary, then $D_{K^*}=0$ $\iff$ $D_{M}=0$ and
$D_{F^*}=0\iff D_G=0$.
\end{corollary}
Let us give connections between the parametrization of a unitary
block-operator matrix given by \eqref{twee} and \eqref{BLOCKS}.
\begin{proposition}
\label{uncion} \cite[Proposition 4.7]{OAM2009}. Let
\[
\begin{array}{l}
U=
\begin{bmatrix} -KA^*M+D_{K^*}XD_{M}& KD_{A}\cr
D_{A^*}M &A\end{bmatrix}\\
 \qquad\quad=\begin{bmatrix}D&D_{D^*}G\cr FD_D
&-FD^*G+D_{F^*}LD_{G}\end{bmatrix}:
\begin{array}{l}\sM\\\oplus\\\cH\end{array}\to
\begin{array}{l}\sN\\\oplus\\\cH\end{array}
\end{array}
\]
be a unitary operator matrix. Then %$L=A\uphar\ker D_A$ and
\begin{equation}
\label{connect1} D_D=M^*D_{A^*}M,\;\sD_D=\ran
M^*,\;D_{D^*}=KD_AK^*,\; \sD_{D^*}=\ran K,
\end{equation}
%\begin{equation}\label{connect6}\end{equation}
\begin{equation}
\label{conect2} F^*=M^*P_{\sD_{A^*}}, \;F=M\uphar\sD_D,\;
G=KP_{\sD_{A}},\; G^*=K^*\uphar\sD_{D^*},
\end{equation}
%\begin{equation}\label{connect3}\end{equation}
\begin{equation}
\label{connect4}
 GFf=KP_{\sD_{A}}M f,\; f\in\sD_D.
\end{equation}
%\begin{equation}%\label{connect5}\end{equation}
\end{proposition}

\section{Completely non-unitary contractions}  \label{secS}

A contraction $A$ acting in a Hilbert space $\sH$ is called
\textit{completely non-unitary} \cite{SF} if there is no nontrivial
reducing subspace of $A$, on which $A$ generates a unitary operator.
Given a contraction $A$ in $\sH$, then there is a canonical
orthogonal decomposition \cite [Theorem I.3.2]{SF}
\[
\sH=\sH_0\oplus \sH_1, \qquad A=A_0\oplus A_1, \quad A_j=A\uphar
\sH_j, \quad j=0,1,
\]
where $\sH_0$ and $\sH_1$ reduce $A$, the operator $A_0$ is a
completely non-unitary contraction, and $A_1$ is a unitary operator.
Moreover,
\[
\sH_1= \left(\bigcap\limits_{n\ge 1}\ker
D_{A^n}\right)\bigcap\left(\bigcap\limits_{n\ge 1}\ker
D_{A^{*n}}\right).
\]
Since
\begin{equation}
\label{INntersec}
\bigcap\limits_{k=0}^{n-1} \ker(D_{A}A^{k})=\ker D_{A^{n}},\;
\bigcap\limits_{k=0}^{n-1} \ker(D_{A^*}A^{*k})=\ker D_{A^{*n}},
\end{equation}
we get
\begin{equation}
\label{perp}
\begin{array}{l}
\bigcap\limits_{n\ge 1}\ker D_{A^n}=\sH\ominus\cspan\left\{A^{*n}D_A
\sH,\;
n\in\dN_0\right\},\\
\bigcap\limits_{n\ge 1}\ker
D_{A^{*n}}=\sH\ominus\cspan\left\{A^nD_{A^*} \sH,\;
n\in\dN_0\right\}.
\end{array}
\end{equation}
It follows that
\begin{equation}
\label{cu}\begin{array}{l}
 A \;\mbox{is completely
non-unitary}\;\iff\left(\bigcap\limits_{n\ge 1}\ker
D_{A^n}\right)\bigcap\left(\bigcap\limits_{n\ge 1}\ker
D_{A^{*n}}\right)=\{0\}\\
 \iff \cspan\{A^{*n}D_{A},\;A^mD_{A^*},\;n,m\in\dN_0\}=\sH.
\end{array}
\end{equation}
Note that
\[
\ker D_{A}\supset\ker D_{A^2}\supset\cdots\supset\ker
D_{A^n}\supset\cdots,
\]
\[
A\ker D_{A^n}\subset\ker D_{A^{n-1}},\; n=2,3,\ldots.
\]
From \eqref{perp} we get that the subspaces $\bigcap\limits_{n\ge
1}\ker D_{A^n}$ and $\bigcap\limits_{n\ge 1}\ker D_{A^{*n}}$ are
invariant with respect to $A$ and $A^*$, respectively, and the
operators $A\uphar\bigcap\limits_{n\ge 1}\ker D_{A^n}$ and
$A^*\uphar\bigcap\limits_{n\ge 1}\ker D_{A^{*n}}$ are unilateral
shifts, moreover, these operators are the maximal unilateral shifts
contained in $A$ and $A^*$, respectively \cite[Theorem 1.1,
Corollary 1]{DR}. By definition \cite{DR} the operator $A$ contains
a co-shift $V$ if the operator $A^*$ contains the unilateral shift
$V^*$. In accordance with the terminology of \cite{Ball-Coehn}, a
contraction $A$ in $\sH$ is called \textit{completely non-isometric}
(\textit{c.n.i.}) if there is no nonzero invariant subspace for $A$
on which $A$ is isometric. This equivalent to (see
\cite{Ball-Coehn})
\[
\bigcap\limits_{n\ge 1}\ker D_{A^n}=\{0\}.
\]
A contraction $A$ is called \textit{completely non-co-isometric}
(\text{c.n.c.-i.}) if $A^*$ is completely non-isometric. Thus, for a
completely non-unitary contraction $A$ we have
\begin{equation}
\label{SHTCOSHT}
\begin{array}{l}
\bigcap\limits_{n\ge 1}\ker D_{A^n}=\{0\}\iff A\;\mbox{is
c.n.i.}\iff A\;\mbox{does
not contain a unilateral shift},\\
\bigcap\limits_{n\ge 1}\ker D_{A^{*n}}=\{0\}\iff  A\;\mbox{is
c.n.c.-i.}\iff  A^*\;\mbox{does not contain a unilateral shift}.
\end{array}
\end{equation}

If $\tau=\left\{\begin{bmatrix} D&C \cr B&A\end{bmatrix};\sM,\sN,
\sH\right\}$ is a conservative system, then
 $\tau$ is simple if and only if the state space operator $A$ is a
completely non-unitary contraction \cite{Br1,Ball-Coehn}. Moreover,
\[
\sH^c_\tau=\cspan\{A^{n}D_{A^*},\; n\in\dN_0\},\; \sH^0_
\tau=\cspan\{A^{*n}D_{A},\; n\in\dN_0\}.
\]
Let $A$ be a contraction in a separable Hilbert space $\sH$. Suppose
$\ker D_A\ne \{0\}$. Define the subspaces \cite{OAM2009}
\begin{equation}
\label{hnm} \left\{
\begin{array}{l}
\sH_{0,0}:=\sH\\
 \sH_{n,0}=\ker D_{A^n},\;
  \sH_{0,m}:=\ker
D_{A^{*m}},\\
\sH_{n, m}:=\ker D_{A^{n}}\cap \ker D_{A^{*m}},\; m,n\in\dN.
\end{array}
\right.
\end{equation}
Let $P_{n,m}$ be the orthogonal projection in $\sH$ onto
$\sH_{n,m}$. Define the contractions \cite{OAM2009}:
\begin{equation}
\label{Anm}
 A_{n, m}:=P_{n,
m}A\uphar\sH_{n, m}\in\bL(\sH_{n,m}).
\end{equation}
Observe that (see  \cite{OAM2009})
%\begin{itemize}
%\item
the following relations are valid:
\begin{equation}
\label{AKNM}%\left\{\begin{array}{l}
 \ker D_{A^k_{n,m}}=\sH_{n+k,m},\;\\
 \ker D_{A^{*k}_{n,m}}=\sH_{n,m+k}.\;
 %\end{array},\right.
 k=1,2,\ldots,
 \end{equation}
% hold,
%\item
\begin{equation} \label{UE1}
 \left(A_{n,m}\right)_{k,l}=A_{n+k,m+l},
\end{equation}
%\item
the operators $A\uphar\sH_{n,m}\in\bL(\sH_{n,m},\sH_{n-1, m+1})$ are unitary,
$A_{n-1,m}\uphar\sH_{n,m}=A\uphar\sH_{n,m}$, and
\[
A_{n-1,m+1}Af=AA_{n,m}f,\;f\in\sH_{n,m},\; n\ge 1,
\]
i.e., the operators
 \[
A_{n,0},\;A_{n-1,1},\; \ldots, A_{n-k,k},\ldots,A_{0,n}
\]
are unitarily equivalent.
%\end{itemize}
 The relation \eqref{UE1} yields the following picture for the
creation of the operators $A_{n,m}$:

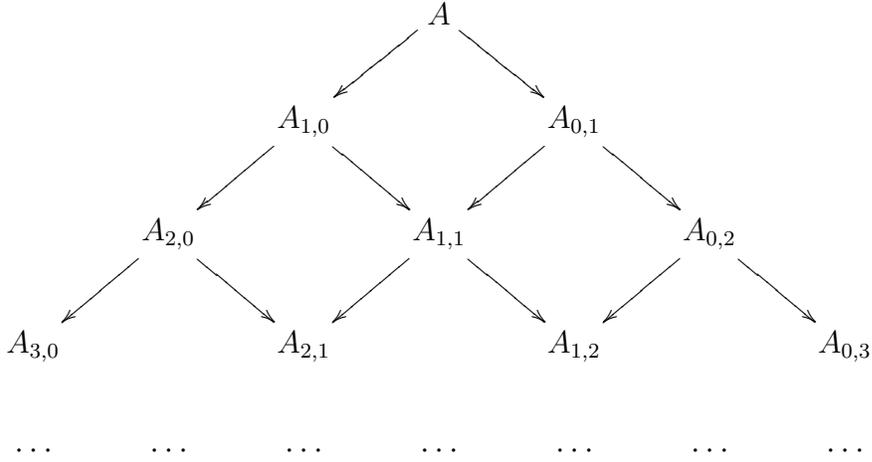
\xymatrix{
&&&A\ar[ld]\ar[rd]\\
&&A_{1,0}\ar[ld]\ar[rd]&&A_{0,1}\ar[ld]\ar[rd]\\
&A_{2,0}\ar[ld]\ar[rd]&&A_{1,1}\ar[ld]\ar[rd]&&A_{0,2}\ar[ld]\ar[rd]\\
A_{3,0}&&
A_{2,1}&&A_{1,2}&&A_{0,3}\\
\cdots&\cdots&\cdots&\cdots&\cdots&\cdots&\cdots}

The process terminates at the $N$-th step if and only if
\[
\begin{array}{l}
\ker D_{A^N}=\{0\}\iff \ker D_{A^{N-1}}\cap\ker
D_{A^*}=\{0\}\iff\ldots\\
\iff\ker D_{A^{N-k}}\cap\ker D_{A^{*k}}=\{0\}\iff\ldots \ker
D_{A^{*N}}=\{0\}.
\end{array}
\]

%\subsection{Defect functions of the Schur class functions}

The following result \cite[Proposition V.4.2]{SF} is needed in the
sequel.
\begin{theorem}\label{Fact}
Let $\sM$ be a separable Hilbert space and let $N(\xi)$,
$\xi\in\dT$,  be an $\bL(\sM)$-valued measurable function such that
$0\le N(\xi)\le I$. Then there exist a Hilbert space $\sK$ and an
outer function $\varphi(\lambda) \in {\bf S}(\sM,\sK)$ satisfying
the following conditions:
\begin{enumerate}
\item $\varphi^*(\xi)\varphi(\xi)\le N^2(\xi)$ a.e.
on $\dT$;

\item if $\wt \sK$ is a Hilbert space and $\wt\varphi(\lambda)\in
{\bf S}(\sM,\wt\sK)$ is such that
$\wt\varphi^*(\xi)\wt\varphi(\xi)\le N^2(\xi)$ a.e.  on $\dT$, then
$\wt\varphi^*(\xi)\wt\varphi(\xi)\le \varphi^*(\xi)\varphi(\xi)$
a.e.  on $\dT$.
\end{enumerate}
Moreover, the function $\varphi(\lambda)$ is uniquely defined up to
a left constant unitary factor.
\end{theorem}

Assume that $\Theta\in {\bf S}(\sM,\sN)$ and denote by
$\varphi_\Theta(\xi)$ and $\psi_\Theta(\xi)$, $\xi \in \dT$ the
outer functions which are solutions of the factorization problem
described in Theorem \ref{Fact} for
$N^2(\xi)=I-\Theta^*(\xi)\Theta(\xi)$ and
$N^2(\bar\xi)=I-\Theta(\bar\xi)\Theta^*(\bar\xi)$, respectively.
Clearly, if $\Theta(\lambda)$ is inner or co-inner, then
$\varphi_\Theta=0$ or $\psi_\Theta=0$, respectively.  The functions
$\f_\Theta(\lambda)$ and $\psi_\Theta(\lambda)$ are called the right
and left \textit{defect functions} (or the \textit{spectral
factors}), respectively, associated with $\Theta(\lambda)$; cf.
\cite{BC, BD, BDFK1, BDFK2, DR}.
 The following result has been established in  \cite[Theorem 1.1, Corollary 1]{DR}
(see also \cite[Theorem 3]{BDFK1}, \cite[Theorem 1.5]{BDFK2}).
\begin{theorem}
\label{DBR} Let $\Theta\in {\bf S}(\sM,\sN)$ and let $
%\[
\tau=\left\{\begin{bmatrix} D&C \cr
B&A\end{bmatrix};\sM,\sN,\sH\right\}$
%\]
be a simple conservative system with transfer function $\Theta$.
Then
\begin{enumerate}
\item the functions $\varphi_\Theta(\lambda)$ and $\psi_\Theta(\lambda)$
take the form
\[
\begin{array}{l}
 \varphi_\Theta(\lambda)=P_\Omega(I_\sH-\lambda A)^{-1}B,\\
\psi_\Theta(\lambda)=C(I_\sH-\lambda A)^{-1}\uphar\Omega_*,
\end{array}
\]
where
\[
\Omega=(\sH^o_\tau)^\perp\ominus A(\sH^o_\tau)^\perp ,\;
\Omega_*=(\sH^c_\tau)^\perp\ominus A^*(\sH^c_\tau)^\perp;
\]
%and $P_\Omega$ is the orthogonal projector from $\sH$ onto $\Omega$;
\item
$\varphi_\Theta(\lambda)=0$ ($\psi_\Theta(\lambda)=0$) if and only
if the system $\tau$ is observable (controllable).
\end{enumerate}
\end{theorem}
The defect functions play an essential role in the problems of the
system theory, in particular, in the problem of similarity and
unitary similarity of the minimal passive systems with equal
transfer functions \cite{ArNu1}, \cite{ArNu2} and in the problem of
\textit{optimal} and $(*)$ \textit{optimal} realizations of the
Schur function \cite{Arov}, \cite{ArKaaP}.
\section{Conservative realization of the Schur
algorithm} \label{CRSCAL}
\begin{theorem}
\label{Schuriso} \cite{OAM2009}. 1) Let the system
\[
\tau=\left\{\begin{bmatrix}D&D_{D^*}G\cr FD_D&-FD^*G+D_{F^*}LD_G
\end{bmatrix};\sM,\sN,\sH\right\}
\]
be conservative and simple and let $\Theta$ be its transfer
function. Suppose that the first associated function $\Theta_1$
is non-unitary constant. Then the systems
\begin{equation}
\label{newzeta}
\begin{array}{l}
\zeta_1=\left\{\begin{bmatrix} GF& G\cr LD_GF&
LD_G\end{bmatrix};\sD_D,\sD_{D^*},\sD_{F^*}\right\},\\
\zeta_2=\left\{\begin{bmatrix} GF& GL\cr D_GF&
D_GL\end{bmatrix};\sD_D,\sD_{D^*},\sD_{G}\right\}
\end{array}
\end{equation}
are conservative and simple and their transfer functions are equal
to $\Theta_1$.
%\end{theorem}
%\begin{theorem}
%\label{IITER} \cite{OAM2009}.

2) Let $\Theta\in{\bf S}(\sM,\sN)$, $\Gamma_0=\Theta(0)$ and let
$\Theta_1$ be the first associated function. Suppose
\[
\tau=\left\{\begin{bmatrix}\Gamma_0&C\cr
B&A\end{bmatrix};\sM,\sN,\sH\right\}
\]
is a simple conservative system with transfer function $\Theta$.
Then the simple conservative systems
\begin{equation}\label{RealFirst}
\begin{array}{l}
\zeta_{1}=\left\{\begin{bmatrix}D^{-1}_{\Gamma^*_0}C(D^{-1}_{\Gamma_0}B^*)^*&D^{-1}_{\Gamma^*_0}C\uphar\ker
D_{A^*}\cr AP_{\ker D_A}D^{-1}_{A^*}B&P_{\ker D_{A^*}}A\uphar\ker
D_{A^*}\end{bmatrix};\sD_{\Gamma_0},\sD_{\Gamma^*_0},\ker
D_{A^*}\right\},\\
\zeta_{2}=\left\{\begin{bmatrix}D^{-1}_{\Gamma^*_0}C(D^{-1}_{\Gamma_0}B^*)^*&
D^{-1}_{\Gamma^*_0}CA\uphar\ker {D_A}\cr P_{\ker
D_A}D^{-1}_{A^*}B&P_{\ker D_A}A\uphar\ker
D_A\end{bmatrix};\sD_{\Gamma_0},\sD_{\Gamma^*_0},\ker D_A\right\}
 \end{array}
\end{equation}
have transfer functions $\Theta_1$. Here the operators
$D^{-1}_{\Gamma_0},$ $D^{-1}_{\Gamma^*_0},$ and $D^{-1}_{A^*}$ are
the Moore--Penrose pseudo-inverses.
\end{theorem}
\begin{theorem}
\label{ITERATES11} \cite{OAM2009}. Let $\Theta\in {\bf S}(\sM,\sN)$
and let
%\[
$\tau_0=\left\{\begin{bmatrix}\Gamma_0&C\cr
B&A\end{bmatrix};\sM,\sN,\sH\right\} $%\]
 be a simple conservative
realization of $\Theta$. Then
 for each $n\ge 1$ the unitarily equivalent simple conservative systems
\begin{equation}
\label{taun}
\begin{array}{l}
\tau^{(k)}_{n}=\left\{\begin{bmatrix}\Gamma_n&D^{-1}_{\Gamma^*_{n-1}}
\cdots D^{-1}_{\Gamma^*_{0}}(CA^{n-k})\cr
A^k\left(D^{-1}_{\Gamma_{n-1}}\cdots D^{-1}_{\Gamma_{0}}
\left(B^*\uphar\sH_{n,0}\right)\right)^*&A_{n-k,k}\end{bmatrix};
\sD_{\Gamma_{n-1}},\sD_{\Gamma^*_{n-1}}, \sH_{n-k,k}\right\},\\
k=0,1,\ldots,n \end{array}
\end{equation}
are realizations of the $n$-th associated function $\Theta_n$ of the
function $\Theta$.
 Here the operator
\[
B_n=\left(D^{-1}_{\Gamma_{n-1}}\cdots D^{-1}_{\Gamma_{0}}
\left(B^*\uphar\sH_{n,0}\right)\right)^*\in\bL(\sD_{\Gamma_{n-1}},\sH_{n,0})
\]
is the adjoint to the operator
\[
D^{-1}_{\Gamma_{n-1}}\cdots D^{-1}_{\Gamma_{0}}
\left(B^*\uphar\sH_{n,0}\right)\in\bL(\sH_{n,0},\sD_{\Gamma_{n-1}}).
\]
\end{theorem}
Notice that the systems
$\tau_n^{(0)},\tau_n^{(1)},\ldots,\tau_n^{(n)} $ are unitarily
similar. In addition
\[
\left(\tau_n^{(k)}\right)_m^{(l)}=\tau_{n+k}^{(k+l)},\;
k=0,1,\ldots,n,\; l=0,\ldots m.
\]
This property can be illustrated by the following picture

 
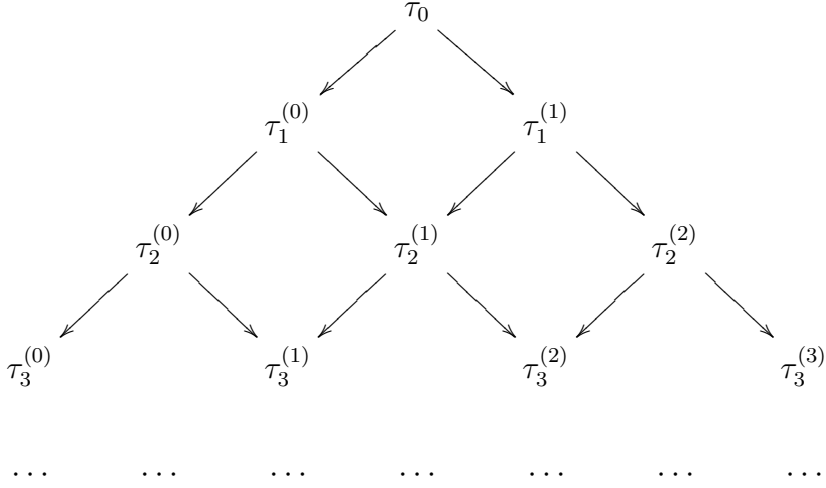
\xymatrix{
&&&\tau_0\ar[ld]\ar[rd]\\
&&\tau_1^{(0)}\ar[ld]\ar[rd]&&\tau_1^{(1)}\ar[ld]\ar[rd]\\
&\tau_2^{(0)}\ar[ld]\ar[rd]&&\tau_2^{(1)}\ar[ld]\ar[rd]&&\tau_2^{(2)}\ar[ld]\ar[rd]\\
\tau_3^{(0)}&&
\tau_3^{(1)}&&\tau_3^{(2)}&&\tau_3^{(3)}\\
\cdots&\cdots&\cdots&\cdots&\cdots&\cdots&\cdots}

\section{Some new properties of the Kre\u{\i}n shorted operators} \label{KRSHOP}
  The next statement is well known.
\begin{proposition} \cite{AT}.
\label{short}Let $\cK$ be a subspace in $\cH$. Then
\begin{enumerate}
\item if $S_1$ and $S_2$ are nonnegative selfadjoint operators then
$$\left(S_1+S_2\right)_\cK\ge \left(S_1\right)_\cK+\left(S_2\right)_\cK;$$
\item
$S_1\ge S_2\ge 0$ $\Rightarrow $ $\left(S_1\right)_\cK\ge
\left(S_2\right)_\cK$;
\item
if $\{S_n\}$ is a nonincreasing sequence of nonnegative bounded
selfadjoint operators and $S=s-\lim\limits_{n\to\infty}S_n$ then
$$s-\lim\limits_{n\to\infty} \left(S_n\right)_\cK=S_\cK.$$
\end{enumerate}
\end{proposition}
Let $\cK^\perp=\cH\ominus \cK$. Then a bounded selfadjoint operator
$S$ has the block-matrix form
\[%\begin{equation}\label{blockS}
S=\begin{pmatrix}S_{11}&S_{12}\cr S^*_{12}&S_{22}
\end{pmatrix}:\begin{array}{l}\cK\\\oplus\\\cK^\perp \end{array}\to
\begin{array}{l}\cK\\\oplus\\\cK^\perp \end{array}.
\]%\end{equation}
 It is well known (see \cite{KrO}) that

 \textit{the
operator $S$ is nonnegative if and only if
\begin{equation}
\label{POZ} S_{22}\ge 0,\; \ran S^*_{12}\subset\ran
S^{1/2}_{22},\,\; S_{11}\ge \left(S^{-1/2}_{22}S^*_{12}\right)^*\left(S^{-1/2}_{22}S^*_{12}\right)
\end{equation}
and the operator $S_\cK$ is given by the block matrix
\begin{equation}
\label{shormat1}
S_\cK=\begin{pmatrix}S_{11}-\left(S^{-1/2}_{22}S^*_{12}\right)^*\left(S^{-1/2}_{22}S^*_{12}\right)&0\cr
0&0\end{pmatrix}.
\end{equation}
} If $S^{-1}_{22}\in\bL(\cK^\perp)$ then the right hand side of
\eqref{shormat1} is of the form
\[
\begin{pmatrix}S_{11}-S_{12}S^{-1}_{22}S^*_{12}&0\cr
0&0\end{pmatrix}
\]
and is called the \textit{Schur complement} of the
matrix $S$. From \eqref{shormat1} it follows that
\[
S_\cK=0\iff \ran S^*_{12}\subset\ran
S^{1/2}_{22}\quad\mbox{and}\quad
S_{11}=\left(S^{-1/2}_{22}S^*_{12}\right)^*\left(S^{-1/2}_{22}S^*_{12}\right).
\]
\begin{proposition}
\label{vspom} Let a bounded nonnegative self-adjoint operator $\bS$
be given by
\[
\bS=\begin{bmatrix}S&0\cr0 &I
\end{bmatrix}:\begin{array}{l}\cL\\\oplus\\\cM\end{array}\to \begin{array}{l}\cL\\\oplus\\\cM\end{array}
\]
and let $\cK$ be a subspace of $\cL$. Then
\[
\bS_\cK=S_\cK P_\cL.
\]
\end{proposition}
\begin{proof}
The inclusion $\cK\subset\cL$ yields
$$\Omega:=\left\{f\in\cran\bS:\bS^{1/2}f\in\cK\right\}=\left\{f\in\cran S:S^{1/2}f\in \cK\right\}\subset\cL.$$
It follows that
\[
\bS_\cK=\bS^{1/2}P_\Omega\bS^{1/2}=S^{1/2}P_\Omega
S^{1/2}P_\cL=S_\cK P_\cL.
\]
\end{proof}
\begin{proposition}
\label{new1} Let $S$ be a bounded nonnegative selfadjoint operator
in the Hilbert space $\cH$, $P$ be an orthogonal  projection in
$\cH$, and let $\cK$ be a subspace in $\cH$ such that
$\cK\subseteq\ran(P)$. Then
$$\left(PSP\right)_\cK\ge S_\cK.$$
\end{proposition}
\begin{proof}
Let $f\in\cK$. Then by \eqref{Sh1} and taking into account that
$P\cK^\perp\subset\cK^\perp$ we get
\[
\begin{array}{l}
\left(\left(PSP\right)_\cK f,f\right)=\inf\limits_{\f\in \cK^\perp}\left\{\left\|S^{1/2}P(f + \varphi)\right\|^2\right\}\\
=\inf\limits_{\f\in \cK^\perp}\left\{\left\|S^{1/2}(f + P\varphi)\right\|^2\right\}
=\inf\limits_{\psi\in \ran(P)\cap \cK^\perp}\left\{\left\|S^{1/2}(f + \psi)\right\|^2\right\}\\
\ge\inf\limits_{\f\in \cK^\perp}\left\{\left\|S^{1/2}(f + \f)\right\|^2\right\}
=\left(S_\cK f,f\right).
\end{array}
\]
Now the equalities
$$\left(PSP\right)_\cK\uphar \cK^\perp=(S)_\cK\uphar\cK^\perp=0,$$
yield that $\left(PSP\right)_\cK\ge S_\cK.$
\end{proof}
\begin{remark}
\label{nerav} Let $S\ge 0$ be given by a block-operator matrix
\[
S=\begin{pmatrix}S_{11}&S_{12}\cr S^*_{12}&S_{22}
\end{pmatrix}:\begin{array}{l}\cK\\\oplus\\\cK^\perp \end{array}\to
\begin{array}{l}\cK\\\oplus\\\cK^\perp \end{array},
\]
and let  $P=P_\cK$. Then $\left(P_\cK SP_\cK\right)_\cK= P_\cK
SP_\cK=\begin{pmatrix}S_{11}&0\cr 0&0
\end{pmatrix}.$
If $S_{12}\ne 0$, then from \eqref {POZ} and \eqref{shormat1} it
follows that $S_\cK\uphar\cK\ne S_{11}$. Therefore, in general
$\left(PSP\right)_\cK\ne S_\cK$.
\end{remark}
\begin{theorem}
\label{new2}
Let $X$ be a nonnegative contraction in the Hilbert space $\cH$. Assume
\begin{enumerate}
\item there is a sequence $\{X_n\}$ of nonnegative contractions strongly converging to $X$,
\item there is a subspace $\cK$ in $\cH$ such that the sequence of operators $\left\{\left(I-X_n\right)_\cK\right\}$
is non-increasing.
\end{enumerate}
Then
\begin{equation}
\label{Newnew}
 s-\lim\limits_{n\to\infty}\left(I-X_n\right)_\cK\le \left(I-X\right)_\cK.
\end{equation}
\end{theorem}
\begin{proof}
We will use the equality (see \cite[Theorem 2.2]{ARL1})
%\begin{theorem}
%\label{newshort}
\begin{equation}
\label{NEWSH} (I-X)_{\cK}=P_{\cK}-\left((I-X^{1/2}P_{\cK^\perp}
X^{1/2})^{-1/2}X^{1/2}P_{\cK}\right)^* (I-X^{1/2}P_{\cK^\perp}
X^{1/2})^{-1/2}X^{1/2}P_{\cK}.
\end{equation}
for a nonnegative selfadjoint contraction $X$ in
$\cH$.

As is well-known if $B$ is an arbitrary nonnegative selfadjoint
operator, then
\begin{equation}
\label{koren}
 \sup\limits_{g\in\dom
B\setminus\{0\}}\cfrac{|(h,g)|^2}{(Bg,g)}=\left\{\begin{array}{l}||B^{-1/2}h||^2,\;
h\in\ran B^{1/2}\\
+\infty,\; h\notin\ran B^{1/2}\end{array}\right.,
\end{equation}
where $B^{-1/2}$ is the Moore-Penrose pseudo-inverse. Hence,
equality \eqref{NEWSH} for all $X_n$ and each $f,g\in\cH$ yields
\[
\cfrac{|(X^{1/2}_nP_\cK f,g)|^2}{||g||^2-||P_{\cK^\perp} X^{1/2}_n
g||^2}\le ||P_{\cK}f||^2-((I-X_n)_\cK f,f).
\]
Since the sequence of operators
$\left\{\left(I-X_n\right)_\cK\right\}$ is non-increasing, there
exists
\[
W:=s-\lim\limits_{n\to\infty}(I-X_n)_\cK.
\]
Therefore
\[
\cfrac{|(X^{1/2}_nP_\cK
f,g)|^2}{||g||^2-||P_{\cK^\perp} X^{1/2}_n g||^2}\le ||P_{\cK}f||^2-(W f,f).
\]
One can prove that
\[
X=s-\lim\limits_{n\to\infty}X_n\;\Rightarrow \;X^{1/2}=s-\lim\limits_{n\to\infty}X^{1/2}_n
\]
Therefore
\[
\cfrac{|(X^{1/2}P_\cK
f,g)|^2}{||g||^2-||P_{\cK^\perp} X^{1/2} g||^2}\le ||P_{\cK}f||^2-(W f,f).
\]
By virtue \eqref{koren} for $B=I_\cH-X^{1/2}P_{\cK^\perp}X^{1/2}$, we obtain
\[
\left\|(I_\cH-X^{1/2}P_{\cK^\perp} X^{1/2})^{-1/2}X^{1/2}P_{\cK}f\right\|^2\le ||P_{\cK}f||^2-(W f,f).
\]
Now \eqref{NEWSH} yields \eqref{Newnew}.
\end{proof}

\section{Main results} \label{MARE}

\subsection{Shorted operators for defect operators of Toeplitz matrices}
Let $\Theta\in\bS(\sM,\sN)$ and let $
\Theta(\lambda)=\sum\limits_{n=0}^{\infty}\lambda^n C_n $. Recall
that by definition $\wt\Theta(\lambda):=\Theta^*(\bar \lambda),$
$|\lambda|<1.$  We identify $\sM$ ($\sN$, respectively) with the
subspace
\[
\sM\oplus\underbrace{\{0\}\oplus\{0\}\oplus\cdots\oplus\{0\}}_n\quad\left(\sN\oplus\underbrace{\{0\}\oplus\{0\}\oplus\cdots\oplus\{0\}}_n\right)
\]
in $\sM^{n+1}$ ($\sN^{n+1}$), and with
$\sM\oplus\bigoplus\limits_{k=1}^\infty\{0\}\quad\left(\sN\oplus\bigoplus\limits_{k=1}^\infty\{0\}\right)$
in $l_2(\sM)$ ($l_2(\sN)).$
\begin{theorem}
\label{Main44} Let $\Theta\in\bS(\sM,\sN)$ and let
$\{\Gamma_0,\Gamma_1,\cdots\}$ be the Schur parameters of $\Theta$.
Then for each $n$ the relations
\begin{equation}
\label{Vazhno1}
\left(D^2_{T_{\Theta,n}}\right)_{\sM}=D_{\Gamma_0}D_{\Gamma_1}\cdots
D_{\Gamma_{n-1}} D^2_{\Gamma_n} D_{\Gamma_{n-1}}\cdots
D_{\Gamma_1}D_{\Gamma_0}P_\sM
\end{equation}
\begin{equation}
\label{Vazhno2}
\left(D^2_{T_{\wt\Theta,n}}\right)_{\sN}=D_{\Gamma^*_0}D_{\Gamma^*_1}\cdots
D_{\Gamma^*_{n-1}} D^2_{\Gamma^*_n} D_{\Gamma^*_{n-1}}\cdots
D_{\Gamma^*_1}D_{\Gamma^*_0}P_\sN,\\
\end{equation}
hold.
\end{theorem}
\begin{proof} Let
\[\begin{array}{l}
\sM_n:=\underbrace{\{0\}\oplus\{0\}\oplus\cdots\oplus\{0\}}_n\oplus\sM,\;
\sN_n:=\underbrace{\{0\}\oplus\{0\}\oplus\cdots\oplus\{0\}}_n\oplus\sN.\\
%\sM^{(n)}=\sM\underbrace{\{0\}\oplus\{0\}\oplus\cdots\oplus\{0\}}_n,\\
%\sM^{(n)}=\sN\underbrace{\{0\}\oplus\{0\}\oplus\cdots\oplus\{0\}}_n
\end{array}
\]
Clearly, the operator
\begin{equation}
\label{stoepn} S_{\Theta,n}:=\begin{bmatrix}0&0&\ldots&0&C_0\cr
0&0&\ldots&C_0&C_1\cr \vdots&\vdots&\vdots&\vdots&\vdots\cr
C_0&C_{1}&C_{2}&\ldots&C_n
\end{bmatrix}\in\bL\left(\sM^{n+1},\sN^{n+1}\right)
\end{equation}
is a contraction. The matrix $S_{\Theta,n}$ we represent in the
block matrix form
\[
S_{\Theta,n}=\begin{bmatrix}Q_{n-1}&B_{n-1}\cr  B^T_{n-1}&C_n
\end{bmatrix}:\begin{array}{l}\sM^n\\\oplus\\\sM_n\end{array}\to
\begin{array}{l}\sN^n\\\oplus\\\sN_n\end{array},
\]
where
\[
Q_{n-1}=\begin{bmatrix} 0&0&\ldots&0&0\cr
0&0&\ldots&0&C_0\cr\vdots&\vdots&\vdots&\vdots&\vdots\cr
0&C_0&C_1&\ldots&C_{n-2}\end{bmatrix}=\begin{bmatrix}0&0\cr0&S_{\Theta,n-2}
\end{bmatrix},
\]
\[
B_{n-1}=\begin{bmatrix}C_0\cr C_1\cr\vdots\cr C_{n-1}
\end{bmatrix},\;
B^T_{n-1}=\begin{bmatrix}C_0&C_1\ldots&C_{n-1}\end{bmatrix}.
\]
Since $S_{\Theta,n}$ is a contraction, by Theorem \ref{ParContr1} (see \eqref{BLOCKS}) we
have
\[
\begin{array}{l}
B_{n-1}=D_{Q^*_{n-1}}G_{n-1},\; B^T_{n-1}=F_{n-1}D_{Q_{n-1}},\\
C_n=-F_{n-1}Q^*_{n-1}G_{n-1}+D_{F^*_{n-1}}L_{n-1}D_{G_{n-1}}.
\end{array}
\]
In \cite{BC} it is proved that
\[
\begin{array}{l}
||D_{F^*_{n-1}}f||^2=||D_{\Gamma^*_{n-1}}\cdots
D_{\Gamma^*_0}f||^2,\;
f\in\sN,\\
||D_{G_{n-1}}h||^2=||D_{\Gamma_{n-1}}\cdots D_{\Gamma_0}h||^2,\;
h\in\sM.\end{array}
\]
Therefore,
\[
\begin{array}{l}
D_{F^*_{n-1}}f=Y_{n-1}D_{\Gamma^*_{n-1}}\cdots D_{\Gamma^*_0}f,\; f\in\sN,\\
D_{G_{n-1}}h=Z_{n-1}D_{\Gamma_{n-1}}\cdots D_{\Gamma_0}h,\; h\in\sM,
\end{array}
\]
where $Y_{n-1}\in\bL(\sD_{\Gamma^*_{n-1}},\sD_{F^*_{n-1}})$ and
$Z_{n-1}\in\bL(\sD_{\Gamma_{n-1}},\sD_{G_{n-1}})$ are unitary
operators. It follows that
\[
\Gamma_n
=Y^*_{n-1}L_{n-1}Z_{n-1},\;D^2_{\Gamma_n}=Z^*_{n-1}D^2_{L_{n-1}}Z_{n-1}, \;D^2_{\Gamma^*_n}=Y^*_{n-1}D^2_{L^*_{n-1}}Y_{n-1}.
\]
Hence
\begin{equation}
\label{Shortn}
\begin{array}{l}
D_{G_{n-1}}D^2_{L_{n-1}}D_{G_{n-1}}=D_{\Gamma_0}D_{\Gamma_1}\cdots
D_{\Gamma_{n-1}} D^2_{\Gamma_n} D_{\Gamma_{n-1}}\cdots
D_{\Gamma_1}D_{\Gamma_0},\\
D_{F^*_{n-1}}D^2_{L^*_{n-1}}D_{F^*_{n-1}}=D_{\Gamma^*_0}D_{\Gamma^*_1}\cdots
D_{\Gamma^*_{n-1}} D^2_{\Gamma^*_n} D_{\Gamma^*_{n-1}}\cdots
D_{\Gamma^*_1}D_{\Gamma^*_0}.
\end{array}
\end{equation}
Now from Corollary \ref{iscois1} and \eqref{Shortn} it follows that
\begin{equation}
\label{sn1}
\begin{array}{l}
\left(D^2_{S_{\Theta,n}}\right)_{\sM_n}=D_{\Gamma_0}D_{\Gamma_1}\cdots
D_{\Gamma_{n-1}} D^2_{\Gamma_n} D_{\Gamma_{n-1}}\cdots
D_{\Gamma_1}D_{\Gamma_0}P_{\sM_n},\\
\left(D^2_{S^*_{\Theta,n}}\right)_{\sN_n}=D_{\Gamma^*_0}D_{\Gamma^*_1}\cdots
D_{\Gamma^*_{n-1}} D^2_{\Gamma^*_n} D_{\Gamma^*_{n-1}}\cdots
D_{\Gamma^*_1}D_{\Gamma^*_0}P_{\sN_n}.
\end{array}
\end{equation}
Let the operator $J_n\in\bL(\sM^{n+1},\sM^{n+1})$ be given by
\[%\begin{equation}\label{JN}
J_n=\begin{bmatrix}0&0&\ldots&0&I_\sM\cr 0&0&\ldots&I_\sM&0\cr
\vdots&\vdots&\vdots&\vdots&\vdots\cr I_\sM&0&\ldots&0&0
\end{bmatrix}.
\]%\end{equation}
The operator $J_n$ is selfadjoint and unitary, $J_n\sM_n=\sM$, and,
clearly,
\[
S_{\Theta,n}=T_{\Theta,n}J_n,\;
D^2_{S_{\Theta,n}}=J_nD^2_{T_{\Theta,n}}J_n.
\]
It follows that
\[
\left(D^2_{S_{\Theta,n}}\right)_{\sM_n}=J_n\left(D^2_{T_{\Theta,n}}\right)_\sM
J_n.
\]
This relation and \eqref{sn1} lead to \eqref{Vazhno1}. Replacing
$\Theta$ by $\wt\Theta$ we get \eqref{Vazhno2}.
\end{proof}
Notice that the  relation $S^*_{\Theta,n}=J_nT^*_{\Theta_,n}$ yields
\begin{equation}
\label{auxil}
\left(D^2_{T^*_{\Theta,n}}\right)_{\sN_n}=\left(D^2_{S^*_{\Theta,n}}\right)_{\sN_n}=D_{\Gamma^*_0}D_{\Gamma^*_1}\cdots
D_{\Gamma^*_{n-1}} D^2_{\Gamma^*_n} D_{\Gamma^*_{n-1}}\cdots
D_{\Gamma^*_1}D_{\Gamma^*_0}P_{\sN_n}.
\end{equation}
The next statement is an immediate consequence of equalities \eqref{Vazhno1}, \eqref{Vazhno2}, and \eqref{rangeSh}.
\begin{corollary}
\label{compl}
The following conditions are equivalent:
\begin{enumerate}
\def\labelenumi{\rm (\roman{enumi})}
\item $\sM\subset \ran D_{T_{\Theta, n}}$,
 \item $\sN\subset \ran D_{T_{\wt \Theta, n}}$,
\item operators $\Gamma_0,\ldots,\Gamma_n$ have norms less than 1.
\end{enumerate}

\end{corollary}
\begin{theorem}
\label{Main5} The equalities
\begin{equation}
\label{Vazhnrav1}
\left(D^2_{T_{\Theta}}\right)_{\sM}=s-\lim\limits_{n\to\infty}\left(D_{\Gamma_0}D_{\Gamma_1}\cdots
D_{\Gamma_{n-1}} D^2_{\Gamma_n} D_{\Gamma_{n-1}}\cdots
D_{\Gamma_1}D_{\Gamma_0}\right)P_\sM,
\end{equation}
\begin{equation}
\label{Vazhnrav2}
\left(D^2_{T_{\wt\Theta}}\right)_{\sN}=s-\lim\limits_{n\to\infty}\left(D_{\Gamma^*_0}D_{\Gamma^*_1}\cdots
D_{\Gamma^*_{n-1}} D^2_{\Gamma^*_n} D_{\Gamma^*_{n-1}}\cdots
D_{\Gamma^*_1}D_{\Gamma^*_0}\right)P_\sN
\end{equation}
hold.
\end{theorem}
\begin{proof}
Let $P_n$ be the orthogonal projection onto $\sM^{n+1}$ in $l_2(\sM)$
and let $\wh T_{\Theta,n}:=P_nT_\Theta P_n$. Then $T_n$ takes the block operator
matrix form
\[
\wh T_{\Theta,n}=\begin{bmatrix} T_{\Theta,n}&0\cr
0&0\end{bmatrix}:\begin{array}{l}\sM^{n+1}\\\oplus\\(\sM^{n+1})^\perp\end{array}\to
\begin{array}{l}\sN^{n+1}\\\oplus\\(\sN^{n+1})^\perp\end{array},
\]
where
\[
(\sM^{n+1})^\perp=l_2(\sM)\ominus\sM^{n+1},\;(\sN^{n+1})^\perp=l_2(\sN)\ominus\sN^{n+1}.
\]
Hence
\[
D^2_{\wh T_{\Theta,n}}=\begin{bmatrix}D^2_{T_{\Theta,n}}&0\cr 0& I \end{bmatrix}.
\]
Since $\sM\subset\sM^{n+1}$, from Proposition \ref{vspom} it follows
that
\begin{equation}
\label{rav1} \left(D^2_{\wh
T_{\Theta,n}}\right)_\sM\uphar\sM=\left(D^2_{T_{\Theta,n}}\right)_\sM\uphar\sM.
\end{equation}
In addition
\[
||D_{\wh
T_{\Theta,n}}f||^2=||D_{T_{\Theta}}P_nf||^2+||(I-P_n)f||^2+||(I-P_n)T_\Theta
P_nf||^2,\; f\in l_2(\sM).
\]
It follows that
\[
D^2_{\wh T_{\Theta,n}}\ge P_n D^2_{T_\Theta}P_n.
\]
Using Propositions \ref{short} and \ref{new1} we get
\begin{equation}
\label{nerav1}
\left(D^2_{\wh T_{\Theta,n}}\right)_\sM\ge \left(P_n D^2_{T_\Theta}P_n\right)_\sM\ge \left(D^2_{T_\Theta}\right)_\sM.
\end{equation}
Let $X=T^*_\Theta T_\Theta$ and $X_n=\wh T^*_{\Theta,n}\wh T_{\Theta,n}=P_n \wh T^*_{\Theta}P_n\wh T_{\Theta}P_n$, $n=1,2,\ldots.$
Then $X$ and $X_n$ are nonnegative selfadjoint contractions and
\[
s-\lim\limits_{n\to \infty}X_n=X,\; D^2_{\wh T_{\Theta,n}}=I-X_n,\;D^2_{T_\Theta}=I-X.
\]
From \eqref{Vazhno1} and \eqref{rav1} it follows that the sequence $\{D^2_{\wh T_{\Theta,n}}\}_{n=1}^\infty$ is non-increasing.
Therefore, by Theorem \ref{new2} we get that
\[%begin{equation}
%\label{nerav2}
 s-\lim\limits_{n\to\infty}\left(D^2_{\wh T_{\Theta,n}}\right)_\sM \le \left(D^2_{T_\Theta}\right)_\sM.
\]%end{equation}
On the other hand \eqref{nerav1} implies
\[
 s-\lim\limits_{n\to\infty}\left(D^2_{\wh T_{\Theta,n}}\right)_\sM\ge \left(D^2_{T_\Theta}\right)_\sM.
 \]
 Hence
 \[
s-\lim\limits_{n\to\infty}\left(D^2_{\wh T_{\Theta,n}}\right)_\sM= \left(D^2_{T_\Theta}\right)_\sM.
 \]
Now from \eqref{Vazhno1} and \eqref{rav1} we obtain \eqref{Vazhnrav1} and similarly \eqref{Vazhnrav2}.
\end{proof}
Notice that it is proved the equalities
\begin{equation}
\label{limrav}
\begin{array}{l}
s-\lim\limits_{n\to\infty}\left(D^2_{T_{\Theta,
n}}\right)_\sM\uphar\sM=\left(D^2_{T_{\Theta}}\right)_\sM\uphar\sM,\\
s-\lim\limits_{n\to\infty}\left(D^2_{\wt T_{\Theta,
n}}\right)_\sN\uphar\sN=\left(D^2_{\wt
T_{\Theta}}\right)_\sN\uphar\sN.
\end{array}
\end{equation}
\begin{corollary}
\label{compl1}
The following conditions are equivalent:
\begin{enumerate}
\def\labelenumi{\rm (\roman{enumi})}
\item $\sM\subset \ran D_{T_{\Theta}}$,
 \item $\sN\subset \ran D_{T_{\wt \Theta}}$,
\item all Schur parameters $\{\Gamma_k\}_{k=0}^\infty$  of $\Theta$ have norms less than 1.
\end{enumerate}
\end{corollary}
\begin{proof}
 Since $(D^2_{T_{\Theta}})_\sM\uphar\sM\le (D^2_{T_{\Theta_n}})_\sM\uphar\sM$ for each $n$, the condition $\sM\subset \ran D_{T_{\Theta}}$ implies
$\sM\subset \ran D_{T_{\Theta_n}}$ for each $n$. Then equivalence of (i), (ii), (iii) follows from Corollary \ref{compl}.
\end{proof}
Let $H^2(\sM)$, $H^2(\sN)$ be the Hardy spaces \cite{SF}. Denote by
$\cP(\sM)$ $(\cP(\sN))$ the linear manifolds of all polynomial from
$H^2(\sM)$ $(H^2(\sN))$
 and by $\cP_n(\sM)$ $(\cP_n(\sN))$ the linear space
of all polynomials of degree at most $n$. By $P^\sM_n$ $(P^\sN_n)$ we denote the orthogonal projection in
$H^2(\sM)$ $(H^2(\sN))$ onto $\cP_n(\sM)$ $(\cP_n(\sN)$).
\begin{theorem}
\label{inff}
\label{Main55} Let $\Theta\in\bS(\sM,\sN)$ and let
$\{\Gamma_0,\Gamma_1,\cdots\}$ be the Schur parameters of $\Theta$. Then
\[
\begin{array}{l}
\inf\limits_{\begin{array}{l}p\in\cP_n(\sM)\\ p(0)=0\end{array}}\left\{\frac{1}{2\pi}\int\limits_0^{2\pi}
\left(||f-p(e^{it})||^2_\sM-||P^\sN_n\Theta(e^{it})(f-p(e^{it}))||^2_\sN\right) dt\right\}\\
\qquad\qquad\quad=
 ||D_{\Gamma_n} D_{\Gamma_{n-1}}\cdots
D_{\Gamma_1}D_{\Gamma_0}f||^2,\; f\in\sM,
\end{array}
\]
\[
\begin{array}{l}
\inf\limits_{\begin{array}{l}p\in\cP(\sM)\\ p(0)=0\end{array}}\left\{\frac{1}{2\pi}\int\limits_0^{2\pi}
\left(||f-p(e^{it})||^2_\sM-||\Theta(e^{it})(f-p(e^{it}))||^2_\sN\right) dt\right\}\\
\qquad\qquad\quad=
\lim\limits_{n\to\infty} ||D_{\Gamma_n} D_{\Gamma_{n-1}}\cdots
D_{\Gamma_1}D_{\Gamma_0}f||^2,\; f\in\sM,
\end{array}
\]
\[
\begin{array}{l}
\inf\limits_{\begin{array}{l}p\in\cP_n(\sN)\\ p(0)=0\end{array}}\left\{\frac{1}{2\pi}\int\limits_0^{2\pi}
\left(||h-p(e^{it})||^2_\sN-||P^\sM_n\wt\Theta(e^{it})(h-p(e^{it}))||^2_\sM\right) dt\right\}\\
\qquad\qquad\quad=
 ||D_{\Gamma^*_n} D_{\Gamma^*_{n-1}}\cdots
D_{\Gamma^*_1}D_{\Gamma^*_0}h||^2,\; h\in\sN,
\end{array}
\]
\[
\begin{array}{l}
\inf\limits_{\begin{array}{l}p\in\cP_n(\sN)\\ p(0)=0\end{array}}\left\{\frac{1}{2\pi}\int\limits_0^{2\pi}
\left(||h-p(e^{it})||^2_\sN-||\wt\Theta(e^{it})(h-p(e^{it}))||^2_\sM\right) dt\right\}\\
\qquad\qquad\quad=
 \lim\limits_{n\to\infty}||D_{\Gamma^*_n} D_{\Gamma^*_{n-1}}\cdots
D_{\Gamma^*_1}D_{\Gamma^*_0}h||^2,\; h\in\sN.
\end{array}
\]
\end{theorem}
\begin{proof}
One can easily see that
\[
||T_\Theta\vec a||^2_{l_2(\sN)}=||\Theta a||^2_{H^2(\sN)}=\frac{1}{2\pi}\int\limits_{0}^{2\pi}||\Theta(e^{it})a(e^{it})||^2_\sN dt,
\]
where $\vec a=(a_0, a_1,\ldots)\in l_2(\sM)$, $a(z)=\sum_{k=0}^\infty a_k z^k\in H^2(\sM)$.
If $p(z)=p_0+p_1 z+\dots p_n z^n$ and $\vec p=(p_0, p_1,\ldots,p_n)\in\sM^{n+1}$, then
\[
||T_{\Theta,n}\vec p||^2_{\sN^{n+1}}=||P^\sN_n \Theta p||^2_{H^2(\sN)}.
\]
To complete the proof of the theorem we use definition \eqref{Sh1} of the shorted operator and equalities \eqref{Vazhno1}, \eqref{Vazhnrav1},
\eqref{Vazhno1},
and \eqref{Vazhnrav2}.
\end{proof}
\subsection{Schur parameters, controllability, and observability}
\begin{theorem}
\label{Main1} Let $
%\[
\tau=\left\{\begin{bmatrix}D&C\cr B& A\end{bmatrix}; \mathfrak
M,\mathfrak N,\mathfrak H\right\} $
%\]
be a simple conservative system with transfer function $\Theta$. Let
$
%\[
\left\{\Gamma_0,\Gamma_1,\ldots\right\}$
%\]
be the Schur parameters of $\Theta$. Then for each $n$ the relations
\begin{equation}
\label{proj1} ||P_{n,0}Bh||^2=\left\|D_{\Gamma_n}D_{\Gamma_{n-1}}
\cdots D_{\Gamma_0}h\right\|^2,\; h\in\sM,
\end{equation}
\begin{equation}
\label{proj2}
||P_{0,n}C^*f||^2=\left\|D_{\Gamma^*_n}D_{\Gamma^*_{n-1}} \cdots
D_{\Gamma^*_0}f\right\|^2, \; f\in\sN
\end{equation}
hold.
\end{theorem}
\begin{proof}
Clearly $D=\Gamma_0$. The unitary operator
\[
U=\begin{bmatrix}D&C\cr B&
A\end{bmatrix}:\begin{array}{l}\sM\\\oplus\\\sH\end{array}\to
\begin{array}{l}\sN\\\oplus\\\sH\end{array}
\]
admits the representations (see Theorem \ref{ParContr1})
\[
U=
\begin{bmatrix} -KA^*M+D_{K^*}XD_{M}& KD_{A}\cr
D_{A^*}M &A\end{bmatrix}=\begin{bmatrix}\Gamma_0&D_{\Gamma^*_0}G\cr
FD_{\Gamma_0} &-FD^*G+D_{F^*}LD_{G}\end{bmatrix}.
\]
%and \[\]
From Theorem \ref{Schuriso} it follows $\Gamma_1=GF$.
Equality \eqref{connect4} yields
\[
\Gamma_1=KP_{\sD_A}M.
\]
Now taking into account that $F\in\bL(\sD_{\Gamma_0},\sM)$ is
isometry and relation \eqref{conect2}, for $f\in\sD_{\Gamma_0}$ we
get
\[
\begin{array}{l}
||D_{\Gamma_1}f||^2=||f||^2-||KP_{\sD_A}
Mf||^2\\
=||Mf||^2-||P_{\sD_A} Mf||^2=||P_{1,0}Mf||^2.
\end{array}
\]
Hence
\[
||D_{\Gamma_1}D_{\Gamma_0}f||^2=||P_{1,0}MD_{\Gamma_0}f||^2,\;
f\in\sM.
\]
Because $M^*\in\bL(\sD_{\Gamma^*_0},\sN)$ is an isometry, from
\eqref{connect1} we have $MD_{\Gamma_0}=D_{A^*}M=B.$ Thus
\begin{equation}
\label{1step} ||D_{\Gamma_1}D_{\Gamma_0}f||^2=||P_{1,0}Bf||^2,\;
f\in\sM.
\end{equation}
By Theorem \ref{ITERATES11} the simple conservative system
\[
\begin{array}{l}
\tau^{(0)}_{1}=\left\{\begin{bmatrix}\Gamma_1&
D^{-1}_{\Gamma^*_{0}}(CA)\cr \left(D^{-1}_{\Gamma_{0}}
\left(B^*\uphar\sH_{1,0}\right)\right)^*&A_{1,0}\end{bmatrix};
\sD_{\Gamma_{0}},\sD_{\Gamma^*_{0}}, \sH_{1,0}\right\}
\end{array}
\]
has transfer function $\Theta_1$. Let
\[
B_1=\left(D^{-1}_{\Gamma_{0}}
\left(B^*\uphar\sH_{1,0}\right)\right)^*\in\bL(\sD_{\Gamma_0},\sH_{1,0}).
\]
Since the Schur parameters of $\Theta_1$ are
$\{\Gamma_1,\Gamma_2,\ldots\}$, starting from the system
$\tau^{(0)}_1$ and using the equality $(\sH_{1,0})_{1,0}=\sH_{2,0}$
(see \eqref{AKNM}), we obtain similarly to \eqref{1step} the
relation
\[
||D_{\Gamma_2}D_{\Gamma_1}\varphi||^2=||P_{2,0}B_1\varphi||^2,\;
\varphi\in\sD_{\Gamma_0}.
\]
Let us show that
\begin{equation}
\label{B1B} B_1D_{\Gamma_0}=P_{1,0}B.
\end{equation}
Actually for $\varphi\in\sM$ and $\psi\in \sH_{1,0}$ one has
\[
\begin{array}{l}
(B_1D_{\Gamma_0}\varphi,\psi)=(\left(D^{-1}_{\Gamma_{0}}
\left(B^*\uphar\sH_{1,0}\right)\right)^*D_{\Gamma_0}\varphi,\psi)=(D_{\Gamma_0}\varphi,D^{-1}_{\Gamma_{0}}
\left(B^*\uphar\sH_{1,0}\right)\psi)\\
=(\varphi,B^*\psi)=(B\varphi,\psi)=(P_{1,0}B\varphi,\psi).
\end{array}
\]
This proves \eqref{B1B}. Since $\sH_{2,0}\subseteq\sH_{1,0}$, we
have the equalities
\[
P_{2,0}B=P_{2,0}P_{1,0}B=P_{2,0}B_1D_{\Gamma_0},
\]
which lead to
\begin{equation}
\label{2step}
||D_{\Gamma_2}D_{\Gamma_1}D_{\Gamma_0}f||^2=||P_{2,0}Bf||^2,\;
f\in\sM.
\end{equation}
By induction, using the equality $(A_{n,0})_{1,0}=A_{n+1,0}$ (see
\eqref{UE1}), we obtain \eqref{proj1} and similarly \eqref{proj2}.
\end{proof}
Using \eqref{INntersec}, Theorem \ref{ParContr1}, and Corollary
\ref{iscois1}, we may interpret equalities \eqref{proj1} and
\eqref{proj2} as follows
\[
\inf\limits_{\{\f_k\}_{k=0}^{n-1}\subset\sN}\left\{\left\|Bh-\sum_{k=0}^{n-1}A^{*k}C^*\f_k\right\|^2
\right\}=\left\|D_{\Gamma_n}D_{\Gamma_{n-1}} \cdots
D_{\Gamma_0}h\right\|^2,\; h\in\sM,\; n\ge 1,
\]
\[
\inf\limits_{\{\psi_k\}_{k=0}^{n-1}\subset\sM}\left\{\left\|C^*f-\sum_{k=0}^{n-1}A^{k}B\psi_k\right\|^2
\right\}= \left\|D_{\Gamma^*_n}D_{\Gamma^*_{n-1}} \cdots
D_{\Gamma^*_0}f\right\|^2, \; f\in\sN,\; n\ge 1.
\]
\begin{corollary}
\label{Main11} Let $\Theta\in\bS(\sM,\sN)$ and let $\varphi_\Theta$
and $\psi_\Theta$ be the right and the left defect functions of
$\Theta,$ respectively. Then
\begin{equation}
\label{fact1} \varphi^*_\Theta(0)\varphi_\Theta(0)=
s-\lim\limits_{n\to\infty}\left(D_{\Gamma_0}D_{\Gamma_1}\cdots
D_{\Gamma_{n-1}} D^2_{\Gamma_n} D_{\Gamma_{n-1}}\cdots
D_{\Gamma_1}D_{\Gamma_0}\right),
\end{equation}
and
\begin{equation}
\label{fact2}
\psi_\Theta(0)\psi^*_\Theta(0)=s-\lim\limits_{n\to\infty}\left(D_{\Gamma^*_0}D_{\Gamma^*_1}\cdots
D_{\Gamma^*_{n-1}} D^2_{\Gamma^*_n} D_{\Gamma^*_{n-1}}\cdots
D_{\Gamma^*_1}D_{\Gamma^*_0}\right),
\end{equation}
where $\{\Gamma_0,\Gamma_1,\ldots\}$ are the Schur parameters of
$\Theta$.
\end{corollary}
\begin{proof} Let $
%\[
\tau=\left\{\begin{bmatrix}D&C\cr B& A\end{bmatrix}; \mathfrak
M,\mathfrak N,\mathfrak H\right\}$
%\]
be a simple conservative system with transfer function $\Theta$.
Since
\[
\sH_{1,0}\supseteq\sH_{2,0}\supseteq\cdots\sH_{n,0}\supseteq\cdots,
\]
the sequence of orthogonal projections $\{P_{n,0}\}$ strongly
converges to the orthogonal projection $P_{\sH_0}$, where
\[
\sH_0:=\bigcap\limits_{n\ge 1}\sH_{n,0}=(\sH^0_\tau)^\perp.
\]
Therefore
\[
P_{\sH_0}Bh=\lim\limits_{n\to\infty}P_{n,0}Bh,\; h\in\sM.
\]
From \eqref{proj1} it follows
\begin{equation}
\label{limproj1}
||P_{\sH_0}Bh||^2=\lim\limits_{n\to\infty}\left\|D_{\Gamma_n}D_{\Gamma_{n-1}}
\cdots D_{\Gamma_0}h\right\|^2,\; h\in\sM.
\end{equation}
The operator $A\uphar\sH_0$ is a unilateral shift, therefore
$D_Ax=0$ for all $x\in\sH_0$. Since the operator
\[
U=\begin{bmatrix}D&C\cr B& A\end{bmatrix}
\]
is unitary, the operator $B$ is of the form $B=D_{A^*}M$, where
$M^*\in\bL(\sD_{A^*},\sM)$ is isometry. Hence for $h\in\sM $ and
$x\in\sH_0$ one obtains
\[
(P_{\sH_0}Bh, Ax)=(D_{A^*}Mh,Ax)=(Mh, D_{A^*}Ax)=(Mh, AD_Ax)=0.
\]
Thus,
\[
P_{\sH_0}Bh=P_\Omega Bh,\; h\in\sM,
\]
where $\Omega=\sH_0\ominus A\sH_0$. Theorem \ref{DBR} yields that
$P_\Omega Bh=\varphi_\Theta(0)h$ and since the sequence of operators
\[
\left\{D_{\Gamma_0}D_{\Gamma_1}\cdots D_{\Gamma_{n-1}}
D^2_{\Gamma_n} D_{\Gamma_{n-1}}\cdots
D_{\Gamma_1}D_{\Gamma_0}\right\}_{n=0}^\infty
\]
in non-increasing, we obtain \eqref{fact1}, and similarly
\eqref{fact2}.
\end{proof}
Using equalities \eqref{Vazhnrav1}, \eqref{Vazhnrav2},
\eqref{fact1}, and \eqref{fact2} we arrive at the next two
corollaries.
\begin{corollary}
 \label{Main3} 1) The following conditions are equivalent
 \begin{enumerate}
 \def\labelenumi{\rm (\roman{enumi})}
 \item
 the system $\tau$ is observable,
 \item
\begin{equation}
\label{observ}
s-\lim\limits_{n\to\infty}\left(D_{\Gamma_0}D_{\Gamma_1}\cdots
D_{\Gamma_{n-1}} D^2_{\Gamma_n} D_{\Gamma_{n-1}}\cdots
D_{\Gamma_1}D_{\Gamma_0}\right)=0,
\end{equation}
\item $\left(D^2_{T_\Theta}\right)_\sM=0.$
\end{enumerate}

2)  The following conditions are equivalent
 \begin{enumerate}
 \def\labelenumi{\rm (\roman{enumi})}
 \item the system $\tau$ is controllable,
 \item
\begin{equation}
\label{control}
s-\lim\limits_{n\to\infty}\left(D_{\Gamma^*_0}D_{\Gamma^*_1}\cdots
D_{\Gamma^*_{n-1}} D^2_{\Gamma^*_n} D_{\Gamma^*_{n-1}}\cdots
D_{\Gamma^*_1}D_{\Gamma^*_0}\right)=0.
\end{equation}
\item $\left(D^2_{T_{\wt\Theta}}\right)_\sN=0.$
\end{enumerate}
\end{corollary}
\begin{corollary}
\label{Mainn11} Let $A$ be a completely non-unitary contraction in
the Hilbert space $\sH$, let
\[
\Psi_A(\lambda)=\left(-A+\lambda D_{A^*}(I-\lambda
A^*)^{-1}D_A\right)\uphar\sD_A,\; \lambda\in\dD
\]
be the Sz.-Nagy--Foias characteristic function of $A$ \cite{SF}, and
let  $\{\gamma_n\}_{n\ge 0}$ be the Schur parameters of $\Psi_A$.

1) The following conditions are equivalent
\begin{enumerate}
 \def\labelenumi{\rm (\roman{enumi})}
\item $A$ is completely non-isometric,
\item
%\[
$s-\lim\limits_{n\to\infty}\left(D_{\gamma^*_0}D_{\gamma^*_1}\cdots
D_{\gamma^*_{n-1}} D^2_{\gamma^*_n} D_{\gamma^*_{n-1}}\cdots
D_{\gamma^*_1}D_{\gamma^*_0}\right)=0,$
%\]
\item $\left(D^2_{T_{\Psi_A}}\right)_{\sD_{A}}=0$.
\end{enumerate}

2) The following conditions are equivalent
\begin{enumerate}
 \def\labelenumi{\rm (\roman{enumi})}
\item $A$ is completely non-co-isometric,
\item
%\[
$s-\lim\limits_{n\to\infty}\left(D_{\gamma_0}D_{\gamma_1}\cdots
D_{\gamma_{n-1}} D^2_{\gamma_n} D_{\gamma_{n-1}}\cdots
D_{\gamma_1}D_{\gamma_0}\right)=0,$
%\]
\item $\left(D^2_{T_{\Psi_{A^*}}}\right)_{\sD_{A^*}}=0$.
\end{enumerate}
\end{corollary}
\begin{proof} The function $\Psi_A$ is the transfer function of the simple conservative system
\[
\Sigma=\left\{\begin{bmatrix}-A& D_{A^*}\cr D_{A}& A^*\end{bmatrix};
 \sD_{A},\sD_{A^*},\sH\right\}.
\]
Now statements follow from Corollary \ref{Main3}.
\end{proof}

Let us make a few remarks. If $\mu$ is a nontrivial scalar
probability measure on the unit circle $\dT=\{\xi\in \dC: |\xi|=1\}$
( $\mu$ is not supported on a finite set), then with $\mu$ are
associated the monic polynomials $\Phi_n(z,\mu)$ (or $\Phi_n$ if
$\mu$ is understood) orthogonal in the Hilbert space $L^2(\dT,
d\mu)$,
 connected by the Szeg\H{o} recurrence relations
 \begin{equation}\label{szego}
\Phi_{n+1}(z)=z\Phi_n(z)-\bar\alpha_n(\mu)\Phi_n^*(z)
\end{equation}
 with some
complex numbers $\alpha_n(\mu)$, called the {\it Verblunsky
coefficients} \cite{S2}. By definition
\[
\Phi(z)=\sum_{j=0}^n p_jz^j
\Rightarrow \Phi^*(z)=\sum_{j=0}^n \bar p_{n-j}z^j.
\]
 The norm of the polynomials $\Phi_n$ in $L^2(\dT,d\mu)$ can be
computed by:
\[
||\Phi_n||^2=\prod\limits_{j=0}^{n-1}(1-|\alpha_j(\mu)|^2),\qquad n=1,2,\ldots.
\]
A result of Szeg\H{o} --
Kolmogorov -- Krein reads that
\[
\prod\limits_{j=0}^{\infty}(1-|\alpha_j(\mu)|^2)=\exp\left(\frac{1}{2\pi}\int_0^{2\pi}\ln \mu'(t)dt\right),
\]
where $\mu'$ is
the Radon -- Nikodym derivative of $\mu$ with respect to Lebesgue measure $dm$.

%\end{theorem}

Define the {\it
Carath\'eodory function} by
\[%be\label{carat}
F(z)= F(z,\mu):=\int_{\dT}\frac{\zeta+z}{\zeta-z}\,d\mu(\zeta).
\]
 $F$ is an analytic function
in $\dD$ which obeys $\RE F>0$, $F(0)=1$. The {\it Schur function}
is then defined by
\[%be\label{schur}
f(z)=
f(z,\mu):=\frac{1}{z}\,\frac{F(z)-1}{F(z)+1}, \qquad F(z)=\frac{1+z
f(z)}{1-z f(z)}\,,
\]%ee
so it is an analytic function in $\dD$ with
$\sup_{\dD}|f(z)|\leq1$. A one-to-one correspondence can be easily
set up between the three classes (probability measures,
Carath\'eodory and Schur functions). Under this correspondence $\mu$
is trivial, that is, supported on a finite set, if and only if the
associate Schur function is a finite Blaschke product.
Let $\{\gamma_n(f)\}$ be the Schur parameters of $f$. According to Geronimus theorem the equalities
 $\gamma_n(f)=\alpha_n(\mu)$ hold for all $n\ge 0$.

 If a Schur function $f$ is not a finite Blaschke product, the
connection between the non-tangential limit values $f(\zeta)$
 and its Schur parameters $\{\gamma_n\}$ is given by the
formula (see \cite {Boy})
%\begin{equation}
%\label{L1}
\[
\prod\limits_{n=0}^\infty
(1-|\gamma_n|^2)=\exp\left\{\int_\dT\ln(1-|f(\zeta)|^2)dm\right\}.
\]%\end{equation}

Thus, (cf. \cite[Theorem
1.5.7]{S2}):
%\begin{theorem}
%\label{complete}
for any nontrivial probability measure $\mu$ on the
unit circle, the following are equivalent:
\begin{enumerate}
\def\labelenumi{\rm (\roman{enumi})}
\item $\lim\limits_{n\to\infty}||\Phi_n||=0;$
\item
$\prod\limits_{j=0}^{\infty}(1-|\alpha_j(\mu)|^2)=0;$
\item the system
$\{\phi_n=\Phi_n/\|\Phi_n\|\}_{n=0}^\infty$ is the orthonormal basis in $L^2(\dT,
d\mu)$,
\item $\ln \mu'\not\in L^1(\dT)$,
\item $\ln(1-|f(\xi)|^2)\not\in L^1(\dT)$,
\item a simple conservative system with transfer function $f$ is controllable and observable.
\end{enumerate}

In the case, when $f\in\bS(\sM,\sM)$ and the norms of all Schur parameters $\{\Gamma_n\}_{n\ge 0}$ of $f$ are less than 1,
in \cite[Corollary 4.8]{BC} is mentioned that
\[
s-\lim\limits_{n\to\infty}\left(D_{\Gamma^*_0}D_{\Gamma^*_1}\cdots
D_{\Gamma^*_{n-1}} D^2_{\Gamma^*_n} D_{\Gamma^*_{n-1}}\cdots
D_{\Gamma^*_1}D_{\Gamma^*_0}\right)=G_\mu^*(0)G_\mu(0),
\]
where $G_\mu$ is the spectral factor of the operator-valued measure $\mu$ from the integral representation
\[
F(z)=\int_{\dT}\frac{\zeta+z}{\zeta-z}\,d\mu(\zeta)
\]
of the function $F(z)=(1+zf(z))(1-z f(z))^{-1}.$

Analogs of formulas \eqref{Vazhno2} and \eqref{Vazhnrav2} have been
established by G.~Popescu in \cite{Pop2001} for a positive definite
multi-Toeplitz kernel and corresponding generalized Schur
parameters.
%\end{remark}

\subsection{Central solution to the Schur problem}
 Now we return to the Schur problem (see Subsection \ref{SchPr}) with data
$\{C_k\}_{k=0}^N\subset\bL(\sM,\sN)$. The necessary and sufficient
condition of solvability is the contractiveness of the operator
\[
T_N=\begin{bmatrix}C_0&0&0&\ldots&0\cr
C_1&C_0&0&\ldots&0\cr \vdots&\vdots&\vdots&\vdots&\vdots\cr
C_N&C_{N-1}&C_{N-2}&\ldots&C_0\end{bmatrix}.
\]
Suppose that $T_N$ is a contraction and let
$$\Gamma_0=C_0,\;\Gamma_k\in\bL(\sD_{\Gamma_{k-1}},\sD_{\Gamma^*_{k-1}}),\; k=1,\ldots,N$$
be the choice sequence determined by the contractive operator $T_N$.
Notice that equalities \eqref{Vazhno1} and \eqref{Vazhno2} remain
true. Moreover, it follows from \eqref{Vazhno1}, \eqref{Vazhno2}
that
 the sequence of operators $\{(D^2_{T_k})_\sM\uphar \sM\}_{k=0}^N$ and $\{(D^2_{\wt T_k})_\sN\uphar \sN\}_{k=0}^N$
are non-increasing. Below this fact we establish directly.
\begin{proposition}
\label{unique1} Let $C_0, C_1,\ldots, C_N\in\bL(\sM,\sN)$ be the
Schur sequence and let $T_N$ be given by \eqref{trToe}.  Then for
$$
T_k=T_k(C_0,C_1,\ldots, C_k),\;\wt T_k=T_k(C^*_0,C^*_1,\ldots,
C^*_k),\; k=0,1,\ldots,N
$$
 the inequalities
\[
\left(D^2_{T_{k}}\right)_\sM\uphar\sM\ge\left(D^2_{T_{k+1}}\right)_\sM\uphar\sM,\; k=0,\ldots, N-1,
\]
\[
\left(D^2_{\wt T_{k}}\right)_\sN\uphar\sN\ge\left(D^2_{\wt T_{k+1}}\right)_\sN\uphar\sN,\; k=0,\ldots, N-1,
\]
hold.
\end{proposition}
\begin{proof} It is sufficient to prove that $\left(D^2_{T_{k}}\right)_\sM\uphar\sM\ge\left(D^2_{T_{k+1}}\right)_\sM\uphar\sM,$ for
$ k=0,\ldots, N-1.$
The operator $T_{k+1}$ can be represented as follows
\[
T_{k+1}=\begin{bmatrix}T_k&0\cr B_k& C_0 \end{bmatrix},%:\begin{array}{l}\sM^{k+1}\\\oplus\\\sM\end{array}\to
%\begin{array}{l}\sN^{k+1}\\\oplus\\\sN\end{array},
\]
where $B_k=\begin{bmatrix}C_{k+1}&C_{k+2}&\ldots&C_1\end{bmatrix}$. It follows that
\[
D^2_{T_{k+1}}=\begin{bmatrix}D^2_{T_k}-B^*_kB_k&-B^*_kC_0\cr -C^*_0B_k&D^2_{C_0} \end{bmatrix}.
\]
Hence,
\[
D^2_{T_k}\ge D^2_{T_k}-B^*_kB_k=P_{\sM^{k+1}}D^2_{T_{k+1}}\uphar\sM^{k+1},
\]
and from Propositions \ref{short}, \ref{new1}
\[
\left(D^2_{T_k}\right)_\sM\ge \left(D^2_{T_k}-B^*_kB_k\right)_\sM=\left(P_{\sM^{k+1}}D^2_{T_{k+1}}\uphar{\sM^{k+1}}\right)_\sM\ge
 \left(D^2_{T_{k+1}} \right)_\sM\uphar\sM^{k+1}.
\]
Hence,
\[
\left(D^2_{T_k}\right)_\sM\uphar\sM\ge \left(D^2_{T_{k+1}}\right)_\sM\uphar\sM.
\]
\end{proof}
\begin{corollary}
\label{zero} Under conditions of Proposition \ref{unique1} the
equality $$\left(D^2_{T_k}\right)_\sM=0$$
 for some $k\le N-1$ implies
\[
\left(D^2_{T_{k+1}}\right)_\sM=0,\;\left(D^2_{T_{k+2}}\right)_\sM=0,\cdots,\left(D^2_{T_N}\right)_\sM=0.
\]
\end{corollary}
\begin{corollary}
\label{detzero}
Additionally to conditions of Proposition \ref{unique1} suppose that $\sM$ and $\sN$ are one-dimensional.
Then the following conditions are equivalent:
\begin{enumerate}
\def\labelenumi{\rm (\roman{enumi})}
\item $\det D^2_{T_N}=0,$
\item $\left(D^2_{T_N}\right)_\sM=0.$
\end{enumerate}
\end{corollary}
\begin{proof}
(ii)$\Rightarrow$(i). The equality $\left(D^2_{T_N}\right)_\sM=0$ is equivalent to $\sM\cap\ran D^2_{T_N}=\{0\}$. It follows
that $\ran(D^2_{T_N})\ne \sM^{N+1}$. Hence $\det D^2_{T_N}=0.$

Let us prove (i)$\Rightarrow$(ii). Let $m\le N-1$ is such that $\det D^2_{T_m}\ne 0$ and $\det D^2_{T_{m+1}}=0$.
The matrix $T_{m+1}$ takes the form
\[
T_{m+1}=\begin{bmatrix}C_0& 0\cr X_m&T_{m} \end{bmatrix},
\]
where
\[
X_m=\begin{bmatrix}C_1\cr \vdots\cr C_{m+1} \end{bmatrix}.
\]
Then
\[
D^2_{T_{m+1}}=\begin{bmatrix}1-C^*_0C_0 -X^*_mX_m& -X^*_m T_m\cr
-T^*_m X_m& D^2_{T_m} \end{bmatrix}.
\]
As is well known
\[
\det D^2_{T_{m+1}}=\det D^2_{T_m}(1-C^*_0C_0 -X^*_mX_m-X^*_mT_mD^{-2}_{T_m}T^*_mX_m)
\]
Since $\det D^2_{T_m}\ne 0$ and $\det D^2_{T_{m+1}}=0$, we get
\[
1-C^*_0C_0 -X^*_mX_m-X^*_mT_mD^{-2}_{T_m}T^*_mX_m=0.
\]
But
\[
1-C^*_0C_0
-X^*_mX_m-X^*_mT_mD^{-2}_{T_m}T^*_mX_m=\left(D^2_{T_{m+1}}\right)_\sM\uphar\sM.
\]
Thus $ \left(D^2_{T_{m+1}}\right)_\sM=0.$  From Corollary \ref{zero}
we obtain $\left(D^2_{T_{N}}\right)_\sM=0.$
\end{proof}
For a contraction $S\in\bL(\sH_1,\sH_2)$ define the M\"obius
transformation as follows
\[
\cM_S(X):=S+D_{S^*}X(I+ S^*X)^{-1}D_S,\; X\in\bL(\sD_S,\sD_{S^*}),
\;-1\in\rho(S^*X).
\]
Suppose that both subspaces $\sD_{\Gamma_N}$ and $\sD_{\Gamma^*_N}$
are non-trivial. Then all solutions to the Schur problem can be
described as follows.
 Let $W$ be an arbitrary function from
$\bS(\sD_{\Gamma_{N}},\sD_{\Gamma^*_{N}})$ . Then define for
$\lambda\in\dD$
\[
\begin{array}{l}
W_1(\lambda)=\cM_{\Gamma_{N}}(\lambda W(\lambda)),\;
W_2(\lambda)=\cM_{\Gamma_{N-1}}(\lambda W_1(\lambda)),\ldots,\\
\qquad W_{N+1}(\lambda)=\cM_{\Gamma_{0}}(\lambda W_{N}(\lambda)).
\end{array}
\]
Due to the Schur algorithm, the function $\Theta(\lambda)=W_{N+1}(\lambda)$ is a solution to the Schur problem.
We can write $\Theta$ as
\begin{equation}
\label{allsol}
\Theta(\lambda)=\cM_{\Gamma_0}\circ\cM_{\Gamma_1}\circ\cdots\circ\cM_{\Gamma_N}(\lambda
W).
\end{equation}
If $ G_0=W(0)\in\bL(\sD_{\Gamma_N},\sD_{\Gamma^*_N}),\;
G_1\in\bL(\sD_{G_0},\sD_{G^*_0}),\ldots $ are the Schur parameters
of $W$, then
\[
\Gamma_0,\Gamma_1,\ldots,\Gamma_N,G_0,G_1,\ldots
\]
are the Schur parameters of $\Theta$. This procedure, using the
Redhefer product, leads to the representation of all solutions by
means of fractional-linear transformation of $W$ \cite{BC, FoFr}. We
note also that all solutions to the Schur problem can be represented
as transfer functions of simple conservative systems having
block-operator CMV matrices \cite{ArlMFAT2009} constructed by means
of the choice sequence
$\Gamma_0,\Gamma_1,\ldots,\Gamma_N,G_0,G_1,\ldots.$

 Apart from $T_N$ we will consider
the operator
\[
\wt T_N=\begin{bmatrix}C^*_0&0&0&\ldots&0\cr
C^*_1&C^*_0&0&\ldots&0\cr \vdots&\vdots&\vdots&\vdots&\vdots\cr
C^*_N&C^*_{N-1}&C^*_{N-2}&\ldots&C^*_0\end{bmatrix}.
\]
Now we describe one step lifting of the Toeplitz matrix by means of
Kre\u\i n shorted operators $(\left(D^2_{T_N}\right)_\sM$ and
$\left(D^2_{\wt T_N}\right)_\sN.$

\begin{proposition}
\label{KrShortforT} The Kre\u in shorted operators
$\left(D^2_{T_N}\right)_\sM$ and $\left(D^2_{\wt T_N}\right)_\sN$
are of the forms
\begin{equation}
\label{frm1} \left(D^2_{T_N}\right)_\sM\uphar \sM=I-C^*_0C_0-
\left(D^{-1}_{T^*_{N-1}}\begin{bmatrix}C_1\cr \vdots\cr
C_N\end{bmatrix}\right)^* D^{-1}_{T^*_{N-1}}\begin{bmatrix}C_1\cr
\vdots\cr C_N\end{bmatrix}
\end{equation}
\begin{equation}
\label{frm2} \left(D^2_{\wt T_N}\right)_\sN\uphar\sN=I-C_0C^*_0-
\left(D^{-1}_{T_{N-1}}\begin{bmatrix}C^*_N\cr \vdots\cr
C^*_1\end{bmatrix}\right)^* D^{-1}_{T_{N-1}}\begin{bmatrix}C^*_N\cr
\vdots\cr C^*_1\end{bmatrix}.
\end{equation}
Here $D^{-1}_{T^*_{N-1}}$ and $D^{-1}_{T_{N-1}}$ are the
Moore--Penrose pseudo-inverses.
\end{proposition}
\begin{proof}
For $T_N$ we have block-matrix representation
\[
T_N=\begin{bmatrix} C_0 &0\cr
B_N&T_{N-1}\end{bmatrix}:\begin{array}{l}\sM\\\oplus\\\sM^{N}\end{array}\to
\begin{array}{l}\sN\\\oplus\\\sN^{N}\end{array},
\]
where $B_N=\begin{bmatrix} C_1 \cr\vdots\cr C_N\end{bmatrix}.$ It
follows that
\[
D^2_{T_N}=\begin{bmatrix}I-\sum\limits_{k=0}^N C^*_kC_k &-B^*_N
T_{N-1}\cr-T^*_{N-1} B_N&D^2_{T_{N-1}}\end{bmatrix}
\]
Due to \eqref{shormat1} one has
\[
\left(D^2_{T_N}\right)_\sM\uphar \sM=I-\sum\limits_{k=0}^N C^*_kC_k
-(D^{-1}_{T_{N-1}}T^*_{N-1}B_N)^*(D^{-1}_{T_{N-1}}T^*_{N-1}B_N).
\]
Since
\[
\lim\limits_{x\downarrow
1}((xI-T^*_{N-1}T_{N-1})^{-1}T^*_{N-1}B_Nf,T^*_{N-1}B_Nf)=||(D^{-1}_{T_{N-1}}T^*_{N-1}B_N)f||^2
,\; f\in\sM,
\]
and
\[
\begin{array}{l}
((xI-T^*_{N-1}T_{N-1})^{-1}T^*_{N-1}B_Nf,T^*_{N-1}B_Nf)\\
=-||B_Nf||^2+x||(xI-T^*_{N-1}T_{N-1})^{-1/2}B_Nf||^2,
\end{array}
\]
we obtain \eqref{frm1}.

The operator $T^*_N$ can be represented as follows
\[
T^*_N=\begin{bmatrix} T^*_{N-1}&\wh B_N \cr 0&C^*_0
\end{bmatrix}:\begin{array}{l}\sN^{N}\\\oplus\\\sN_N\end{array}\to
\begin{array}{l}\sM^N\\\oplus\\\sM_N\end{array},
\]
where $\wh B_N=\begin{bmatrix} C^*_N \cr\vdots\cr
C^*_1\end{bmatrix}$. Recall that
\[
\sM_N:=\underbrace{\{0\}\oplus\{0\}\oplus\cdots\oplus\{0\}}_N\oplus\sM,\;
\sN_N:=\underbrace{\{0\}\oplus\{0\}\oplus\cdots\oplus\{0\}}_N\oplus\sN.
\]
Then
\[
D^2_{T^*_N}=\begin{bmatrix}D^2_{T^*_{N-1}}&-T_{N-1}\wh B_N\cr -\wh
B^*_NT^*_{N-1}&I-\sum\limits_{k=0}^N C_kC^*_k
\end{bmatrix}.
\]
As above we obtain
\[
\left(D^2_{T^*_N}\right)_{\sN_{N}}\uphar\sN_{N}=I-C_0C^*_0-
\left(D^{-1}_{T_{N-1}}\wh B_N\right)^* D^{-1}_{T_{N-1}}\wh B_N.
\]
Therefore \eqref{frm2} follows from \eqref{Vazhno2} and
\eqref{auxil}.
\end{proof}

\begin{theorem}
\label{central} Let the data $C_0, C_1,\ldots, C_N\in\bL(\sM,\sN)$
be the Schur sequence. Then the formula
\begin{equation}
\label{centr} C_{N+1}= %\stackrel{0}
\dot{C}_{N+1} + \left(\left(D^2_{\wt
T_N}\right)_\sN\uphar\sN\right)^{1/2}
Y\left(\left(D^2_{T_N}\right)_\sM\uphar\sM\right)^{1/2},
\end{equation}
where
\begin{equation}
\label{centr1}
%\stackrel{0}
\dot{C}_{N+1}=-\left(%\begin{bmatrix}
D^{-1}_{T_{N-1}}\begin{bmatrix}C^*_N\cr C^*_{N-1}\cr\vdots\cr
C^*_1\end{bmatrix}%\cr C^*_0\end{bmatrix}
\right)^*T^*_{N-1}%\begin{bmatrix}C_0\cr
 D^{-1}_{T^*_{N-1}}\begin{bmatrix}C_1\cr C_2\cr\vdots\cr
C_N\end{bmatrix}
%\end{bmatrix}
\end{equation}
and $Y$ is an arbitrary contraction from $\bL\left(\cran
\left((D^2_{ T_N})_\sM\right),\cran \left((D^2_{\wt
T_N})_\sN\right)\right),$
 describes all Schur sequences $\{C_0,\ldots, C_N,C_{N+1}\}$.
\end{theorem}
\begin{proof} Represent the matrix
\[
T_{N+1}=\begin{bmatrix}C_0&0&0&\ldots&0&0\cr C_1&C_0&0&\ldots&0&0\cr
\vdots&\vdots&\vdots&\vdots&\vdots&\vdots\cr
C_N&C_{N-1}&C_{N-2}&\ldots&C_0&0 \cr
C_{N+1}&C_{N}&C_{N-1}&\ldots&C_1&C_0\end{bmatrix}\in\bL(\sM^{N+2},\sN^{N+2})
\]
 in the form
\[
T_{N+1}=\begin{bmatrix}B&D\cr A&C
\end{bmatrix}:\begin{array}{l}\sM\\\oplus\\\sM^{N+1}\end{array}\to\begin{array}{l}\sN^{N+1}\\\oplus\\\sN\end{array}
\]
with
\[
\begin{array}{l}
B=\begin{bmatrix}C_0\cr C_1\cr\vdots\cr C_N\end{bmatrix},\;
D=\begin{bmatrix}0&0&0&\ldots&0&0\cr C_0&0&0&\ldots&0&0\cr
\vdots&\vdots&\vdots&\vdots&\vdots&\vdots\cr
C_{N-1}&C_{N-2}&C_{N-3}&\ldots&C_0&0
\end{bmatrix},
\\
A=\begin{bmatrix} C_{N+1} \end{bmatrix},\; C=\begin{bmatrix}C_N&
C_{N-1}&\ldots&C_0  \end{bmatrix}.
\end{array}
\]
On the other hand
\[
D=\begin{bmatrix}0&0\cr T_{N-1}&0
\end{bmatrix}:\begin{array}{l}\sM^N\\\oplus\\\sM\end{array}\to\begin{array}{l}\sN\\\oplus\\\sN^N\end{array}.
\]
The operator $D$ is a contraction and
\[
D_D=\begin{bmatrix} D_{T_{N-1}}&0\cr 0&I
\end{bmatrix},\;D_{D^*}=\begin{bmatrix} I&0\cr 0&D_{T^*_{N-1}}
\end{bmatrix}.
\]
From Theorem \ref{ParContr1} it follows that $T_{N+1}$ is a
contraction if and only if $A$ is of the form (see Section
\ref{parcontmatr})
\[
A=-VD^*U+D_{V^*}YD_{U},
\]
where $C=VD_{D},$ $B=D_{D^*}U,$ $Y\in\bL(\sD_{U},\sD_{V^*}),$
$||Y||\le 1$. Thus,
\[
A=-(D^{-1}_{D}C^*)^*D^*D^{-1}_{D^*}B+D_{V^*}YD_{U}.
\]
We have
\[
V^*= D^{-1}_D C^*=
\begin{bmatrix}D^{-1}_{T_{N-1}}\begin{bmatrix}C^*_N\cr
C^*_{N-1}\cr\vdots\cr C^*_1\end{bmatrix}\cr  C^*_0
\end{bmatrix},\;U=D^{-1}_{D^*}B=\begin{bmatrix} C_0\cr
D^{-1}_{T^*_{N-1}}\begin{bmatrix}C_1\cr C_2\cr\vdots\cr
C_N\end{bmatrix}
\end{bmatrix}.
\]
Then
\begin{equation}
\label{dg} D^2_U=I-C^*
C_0-\left(D^{-1}_{T^*_{N-1}}\begin{bmatrix}C_1\cr C_2\cr\vdots\cr
C_N\end{bmatrix} \right)^*D^{-1}_{T^*_{N-1}}\begin{bmatrix}C_1\cr
C_2\cr\vdots\cr C_N\end{bmatrix},
\end{equation}
\begin{equation}
\label{df*} D^2_{V^*}=I-C_0C^*_0- \left(
D^{-1}_{T_{N-1}}\begin{bmatrix}C^*_N\cr C^*_{N-1}\cr\vdots\cr
C^*_1\end{bmatrix}\right)^* D^{-1}_{T_{N-1}}\begin{bmatrix}C^*_N\cr
C^*_{N-1}\cr\vdots\cr C^*_1\end{bmatrix},
\end{equation}
and
\[
\begin{array}{l}
-VD^*U=-\left(\begin{bmatrix}D^{-1}_{T_{N-1}}\begin{bmatrix}C^*_N\cr
C^*_{N-1}\cr\vdots\cr C^*_1\end{bmatrix}\cr C^*_0
\end{bmatrix} \right)^*D^*\begin{bmatrix} C_0\cr
D^{-1}_{T^*_{N-1}}\begin{bmatrix}C_1\cr C_2\cr\vdots\cr
C_N\end{bmatrix}
\end{bmatrix}\\
\qquad\qquad=-\left(%\begin{bmatrix}
D^{-1}_{T_{N-1}}\begin{bmatrix}C^*_N\cr C^*_{N-1}\cr\vdots\cr
C^*_1\end{bmatrix}%\cr C^*_0\end{bmatrix}
\right)^*T^*_{N-1}%\begin{bmatrix}C_0\cr
 D^{-1}_{T^*_{N-1}}\begin{bmatrix}C_1\cr C_2\cr\vdots\cr
C_N\end{bmatrix}.
\end{array}
\]
From \eqref{frm1} and \eqref{frm2} we get \eqref{centr1}.
\end{proof}
\begin{remark}
For finite dimensional $\sM$ and $\sN$ formulas \eqref{centr1},
\eqref{dg}, and \eqref{df*} can be found in \cite{DFK}.
\end{remark}
Define consequentially the operators $\dot{C}_{N+1},$
$\dot{C}_{N+2},$ $\ldots$ by means of \eqref{centr1} using
$\dot{T}_{N},$ $\dot{T}_{N+1}$, $\dots$. The solution
$$\Theta_0(\lambda)=\sum\limits_{k=0}^N \lambda^kC_k+
\sum\limits_{n=1}^\infty \lambda^{N+n} \dot{C}_{N+n}$$ of the Schur
problem with data $\{C_k\}_{k=0}^N$ is called the \textit{central
solution} \cite{DFK}.
\begin{theorem}
\label{crcent} Let the data $C_0, C_1,\ldots, C_N\in\bL(\sM,\sN)$ be
the Schur sequence. Then the following statements are equivalent:
\begin{enumerate}
\def\labelenumi{\rm (\roman{enumi})}
\item $\Theta$ is the central solution to the Schur problem,
\item $\left(D^2_{T_\Theta}\right)_\sM\uphar\sM=\left(D^2_{T_N}\right)_\sM\uphar\sM,$
\item $\left(D^2_{\wt T_\Theta}\right)_\sN\uphar\sN=\left(D^2_{\wt
T_N}\right)_\sN\uphar\sN.$
\end{enumerate}
\end{theorem}
\begin{proof}
Let $\Theta$ be a solution to the Schur problem,
$\Theta(\lambda)=\sum\limits_{k=0}^\infty\lambda^kC_k$. Then
$C_{N+1}$ is  given by \eqref{centr} with some contraction $Y$. For
corresponding Toeplitz operators $T_{N+1}$, $\wt T_{N+1}$ from
Proposition \ref{unique1} we obtain
\[\begin{array}{l}
\left(D^2_{T_{N+1}}\right)_\sM=\left(\left(D^2_{T_N}\right)_\sM\right)^{1/2}D^2_Y\left(\left(D^2_{T_N}\right)_\sM\right)^{1/2}P_\sM,\\
\left(D^2_{\wt T_{N+1}}\right)_\sN=\left(\left(D^2_{\wt
T_N}\right)_\sN\right)^{1/2}D^2_{Y^*}\left(\left(D^2_{\wt
T_N}\right)_\sN\right)^{1/2}P_\sN.
\end{array}
\]
Since $Y=0$ corresponds to $\dot{C}_{N+1}$, we get
\[
\left(D^2_{\dot{T}_{N+1}}\right)_\sM\uphar\sM=\left(D^2_{T_N}\right)_\sM\uphar\sM.
\]
Similarly
\[
\left(D^2_{\wt {\dot{T}}_{N+1}}\right)_\sN\uphar\sN=\left(D^2_{\wt
T_N}\right)_\sN\uphar\sN.
\]
By induction
\begin{equation}
\label{ravenstvo}
\left(D^2_{\dot{T}_{N+n}}\right)_\sM\uphar\sM=\left(D^2_{{T}_{N}}\right)_\sM\uphar\sM,\;
\left(D^2_{\wt{\dot{T}}_{N+n}}\right)_\sN\uphar\sN=\left(D^2_{\wt{T}_{N}}\right)_\sN\uphar\sN
\end{equation}
for each $n\ge 1$. Hence, if $\Theta=\Theta_0$ is the central
solution, then \eqref{ravenstvo} and \eqref{limrav} imply
\[
\left(D^2_{\dot{T}_{\Theta}}\right)_\sM\uphar\sM=\left(D^2_{\dot{T}_{N+1}}\right)_\sM\uphar\sM,\;
\left(D^2_{\wt
{\dot{T}}_{\Theta}}\right)_\sN\uphar\sN=\left(D^2_{\wt
T_N}\right)_\sN\uphar\sN.
\]
Similarly (iii)$\Rightarrow$ (i) and (ii)$\Rightarrow$ (i).
\end{proof}
 Thus, for the lower
 triangular Toeplitz matrix $T_{\Theta_0}$, corresponding to
 $\Theta_0,$ we obtain the following statement.
\begin{theorem}
 \label{entropy} Let the data $C_0, C_1,\ldots,
C_N\in\bL(\sM,\sN)$ be the Schur sequence. Then the central solution
${\Theta}_0 \in\bS(\sM,\sN)$ is a unique among other solutions
$\Theta$, satisfying
\[
\begin{array}{l}
\left(D^2_{T_{{\Theta}_0}}\right)_\sM\uphar\sM=\max\limits_{\Theta}
\left\{\left(D^2_{T_\Theta}\right)_\sM\uphar\sM\right\}\\
\qquad\iff \left(D^2_{\wt
T_{{\Theta}_0}}\right)_\sN\uphar\sN=\max\limits_{\Theta}
\left\{\left(D^2_{\wt T_\Theta}\right)_\sN\uphar\sN\right\}.
\end{array}
\]
\end{theorem}
Note that the solution $\Theta_0$ is often called \textit{maximal
entropy solution} \cite{FoFr}. If the choice sequence of $T_N$ are
$\Gamma_0=C_0,\Gamma_1,\ldots, \Gamma_N$, then from \eqref{Vazhno1}
and Theorem \ref{crcent}
\[\begin{array}{l}
\left(D^2_{T_{\Theta_0}}\right)_{\sM}=\left(D^2_{T_{N}}\right)_{\sM}=D_{\Gamma_0}D_{\Gamma_1}\cdots
D_{\Gamma_{N-1}} D^2_{\Gamma_N} D_{\Gamma_{N-1}}\cdots
D_{\Gamma_1}D_{\Gamma_0}P_\sM,\\
\left(D^2_{\wt T_{\Theta_0}}\right)_{\sN}=\left(D^2_{\wt
T_{N}}\right)_{\sN}=D_{\Gamma^*_0}D_{\Gamma^*_1}\cdots
D_{\Gamma^*_{N-1}} D^2_{\Gamma^*_N} D_{\Gamma^*_{N-1}}\cdots
D_{\Gamma^*_1}D_{\Gamma^*_0}P_\sN.
\end{array}
\]
From \eqref{ravenstvo} and \eqref{Vazhno1} it follows that the Schur
parameters of $\Theta_0$ are operators
\[
\Gamma_0,\Gamma_1,\ldots, \Gamma_N,
0\in\bL(\sD_{\Gamma_N},\sD_{\Gamma^*_N}),
0\in\bL(\sD_{\Gamma_N},\sD_{\Gamma^*_N}),\ldots.
\]
The function $\Theta_0$ is also given by \eqref{allsol} with
$W(\lambda)=0,\; \lambda\in\dD$. Let
$\{\Theta_n\in\bS(\sD_{\Gamma_{n-1}}\sD_{\Gamma^*_{n-1}})\}_{n\ge
0}$ be functions associated with $\Theta_0$ in accordance with Schur
algorithm. Then
$\Theta_{N+1}=\Theta_{N+2}=\cdots=0\in\bS(\sD_{\Gamma_{N}},\sD_{\Gamma^*_{N}})$.
Let
\[
\tau_0=\left\{\begin{bmatrix}D&C\cr B& A\end{bmatrix}; \mathfrak
M,\mathfrak N,\mathfrak H\right\}
\]
be a simple conservative realization of the central solution
$\Theta_0$. Clearly, $D=C_0=\Gamma_0$. Then by Theorem
\ref{ITERATES11} the simple conservative systems
\[
\begin{array}{l}
\tau^{(k)}_{N+1}=\left\{\begin{bmatrix}0&D^{-1}_{\Gamma^*_{N}}
\cdots D^{-1}_{\Gamma^*_{0}}(CA^{N+1-k})\cr
A^k\left(D^{-1}_{\Gamma_{N}}\cdots D^{-1}_{\Gamma_{0}}
\left(B^*\uphar\sH_{N+1,0}\right)\right)^*&A_{N+1-k,k}\end{bmatrix};
\sD_{\Gamma_{N}},\sD_{\Gamma^*_{N}}, \sH_{N+1-k,k}\right\},\\
k=0,1,\ldots,N+1 \end{array}
\]
realize the function $\Theta_{N+1}=0$. Hence, the  unitarily
equivalent contractions $\left\{A_{N+1-k,k}\right\}_{k=0}^{N+1}$ are
orthogonal sums of unilateral shifts and co-shifts of multiplicities
$\dim\sD_{\Gamma_N}$ and $\dim\sD_{\Gamma^*_N}$, correspondingly
\cite{OAM2009}.
\subsection{Uniqueness solution to the Schur problem}
Here we are interested in the case of uniqueness of the solution to
the Schur problem. The following statement takes place.
\begin{theorem}
\label{OPUNI} Let the data $C_0,C_1,\ldots, C_N\in\bL(\sM,\sN)$ be
the Schur sequence. Then the following statements are equivalent
\begin{enumerate}
\def\labelenumi{\rm (\roman{enumi})}
\item the Schur problem has a unique solution;
\item either $(D^2_{T_N})_\sM=0$ or $(D^2_{\wt T_N})_\sN=0$;
\item either $\sM\cap \ran D_{T_N}=\{0\}$ or $\sN\cap \ran D_{\wt T_N}=\{0\}$.
\end{enumerate}
\end{theorem}
\begin{proof}
We give two proves of the theorem.

\textit{The first proof}. The equivalence of (ii) and (iii) follows
from \eqref{nol}. Let the Schur problem has a unique solution
$\wh\Theta(\lambda)=\sum\limits_{k=0}^N \lambda^kC_k+
\sum\limits_{n=1}^\infty \lambda^{N+n} \wh{C}_{N+n}$. Because
$\{C_0,\ldots, C_N,\wh C_{N+1}\}$ is the Schur sequence, from
\eqref{centr} it follows that $\wh C_{N+1}=\dot{C}_N$ and either
$(D^2_{T_N})_\sM=0$ or $(D^2_{\wt T_N})_\sN=0$. So, (i) implies
(ii). In particular, ewe get that $\wh\Theta=\Theta_0$.

If (ii) holds true, then again from \eqref{centr} we get that
$\Theta_0$ is a unique solution of the Schur problem.

\textit{The second proof}.
 The matrix $T_N$ defines a
sequence of contractions (the choice sequence)
$$\Gamma_0(=C_0),\;\Gamma_1\in\bL(\sD_{\Gamma_0},\sD_{\Gamma^*_0}),\ldots,
\Gamma_N\in\bL(\sD_{\Gamma_{N-1}},\sD_{\Gamma^*_{N-1}}).$$ Suppose
that the Schur problem has a unique solution. Then by Theorem
\ref{SchPr2} one of $\Gamma's$ is an isometry or co-isometry. Assume
$\Gamma_p$ is isometry, where $p\le N$. From Theorem \ref{Main44} it
follows that $(D^2_{T_p})_\sM=0$. Corollary \ref{zero} yields the
equality $(D^2_{T_N})_\sM=0$. If we assume that $\Gamma^*_p$ is
isometry, then similarly we get $(D^2_{\wt T_N})_\sN=0$.

Now suppose $(D^2_{T_N})_\sM=0$. Let $p\le N$ is such that
$(D^2_{T_p})_\sM=0$, but $(D^2_{T_{p-1}})_\sM\ne 0$. Note that in
this case $\sD_{\Gamma_{p-1}}\ne \{0\}$, $\Gamma_p$ is isometry,
$$\sD_{\Gamma_{p}}=\sD_{\Gamma_{p+1}}=\cdots=\sD_{\Gamma_{N-1}}=\{0\},\; \Gamma_{p+1}=\cdots=\Gamma_N=0.$$
It follows that the solution to the Schur problem is unique and is
of the form
\[
\Theta(\lambda)=\cM_{\Gamma_0}\circ\cM_{\Gamma_1}\circ\cdots\circ\cM_{\Gamma_{p-1}}(\lambda\Gamma_p),\;
\lambda\in\dD.
\]
Similarly, the equality $(D^2_{\wt T_N})_\sN=0$ implies the
uniqueness. Thus $(i)\iff (ii)$.
\end{proof}
Observe that $(D^2_{\wt T_N})_\sN=0$ $\iff$
$(D^2_{T^*_N})_{\sN_N}=0$ (see\eqref{auxil}).
\begin{remark}
\label{AAKpap} V.M.~Adamyan, D.Z.~Arov, and M.G.~Kre\u\i n  in
\cite{AAK1} considered the following generalized Nehari--
Carath\'eodory--Fej\'{e}r problem: given a sequence of complex
numbers $\{\gamma_k\}_{1}^\infty,$ find a function $f\in
L_\infty(\dT)$ with principal part $\sum_{k=1}^\infty
\gamma_k\zeta^{-k}$ and with minimal $L_\infty$-norm. By Hehari's
theorem  \cite{Nehari} this problem has a solution if and only if
the Hankel matrix  $\Gamma=||\gamma_{j+k-1}||$ is bounded in $l_2$.
A criteria of the uniqueness solution is established in the form
\cite[Theorem 2.1]{AAK1}
\begin{equation}
\label{b} \lim\limits_{\rho\downarrow ||\Gamma||}\left((\rho^2
I-\Gamma^*\Gamma)^{-1}\vec e ,\vec e\right)=\infty,
\end{equation}
for the vector $\vec e=(1,0,0,\ldots)\in l_2.$ Because
\[
 \lim\limits_{x\uparrow 0}\,\left((B- xI)^{-1}g,g\right)=\left\{
\begin{array}{ll}
    \| B^{-1/2}g\|^2, & g\in\ran B^{1/2},\\
    +\infty,           & g\notin\ran B^{1/2}
\end{array}\right.
\]
for an arbitrary nonnegative selfadjoint operator $B$ ($B^{-1/2}$ is
the Moore-Penrose pseudo-inverse), equality \eqref{b} means that
$$\vec e\notin\ran\left(s^2
I-\Gamma^*\Gamma\right),$$ where $s=||\Gamma||$. Then by \eqref{nol}
one has that \eqref{b} is equivalent to the equality
$$\left(s^2 I-\Gamma^*\Gamma\right)_E=0,$$
where $E=\{\lambda \vec e,\;\lambda\in\dC\}$. The results of
\cite{AAK1} have been extended to the case of operator-valued
 functions in the paper \cite{AAK2} (see also \cite{Peller}).
 The corresponding uniqueness criteria \cite[Theorem 1.3]{AAK2} also takes the limit form similar to the scalar case.
As has been mentioned in Introduction the Schur problem can be
reduced to the above problem and the matrix $s^2 I-\Gamma^*\Gamma$
can be reduced to the square of the defect operator for a lower
triangular Toeplitz matrix.
\end{remark}

\end{document}